\documentclass[oneside,10pt,reqno]{amsart}
\usepackage{amssymb,amsmath,amsthm,bbm,enumerate,hyperref,amsthm,amsaddr,mathrsfs,changepage}
\usepackage[shortlabels]{enumitem}
\allowdisplaybreaks
\theoremstyle{definition}
\newtheorem{definition}{Definition}
\theoremstyle{theorem}
\newtheorem{proposition}[definition]{Proposition}
\newtheorem{lemma}[definition]{Lemma}
\newtheorem{theorem}[definition]{Theorem}
\newtheorem{corollary}[definition]{Corollary}

\numberwithin{equation}{section}
\numberwithin{definition}{section}
\theoremstyle{remark}
\newtheorem{remark}[definition]{Remark}

\newtheorem{question}[definition]{Question}
\newtheorem{example}[definition]{Example}

\def\PP{\mathbb{P}}
\def\AA{\mathcal{A}}

\def\HH{\mathcal{H}}

\def\lll{\mathsf{x}}

\def\ii{\mathsf{i}}
\def\L2{\mathrm{L}^2}

\def\QQ{\mathbb{Q}}

\def\RR{\mathbb{R}}

\def\GG{\mathcal{G}}
\def\FF{\mathcal{F}}

\def\dd{\mathrm{d}}

\def\yyy{\mathsf{y}}

\makeatletter
\@namedef{subjclassname@2020}{%
  \textup{2020} Mathematics Subject Classification}
\makeatother
\begin{document}
\title{Some characterizations for Markov processes at first passage}

\author{Matija Vidmar}
\address{Department of Mathematics, Faculty of Mathematics and Physics, University of Ljubljana}
\email{matija.vidmar@fmf.uni-lj.si}

\begin{abstract}
Suppose $X$ is a Markov process on the real line (or some interval). Do the distributions of its first passage times  downwards (fptd) determine its law?  In this paper we treat some special cases of this question. We prove that if the fptd  process has the law of a subordinator, then necessarily $X$ is a L\'evy process with no negative jumps; specifying  the law of the subordinator determines the law of $X$ uniquely. We further show that, likewise, the classes of continuous-state branching processes and of self-similar processes without negative jumps are also respectively characterised by a certain structure of their fptd distributions; and each member of these classes separately is determined uniquely by the precise family of its fptd laws. The road to these results is paved by (i) the identification of Markov processes without negative jumps in terms of the nature of their fptd laws, and (ii) some general results concerning the identification of the fptd distributions for such processes.
\end{abstract}


\thanks{Support from the Slovenian Research Agency  (project No. N1-0174) is acknowledged. The author thanks Zoran Vondraček, Mihael Perman and Jon Warren for some kind discussions on the topics of this paper. Parts of this text were completed while the author was a Fernandes Fellow with the University of Warwick, he is grateful for its hospitality.}  

\date{\today}

\maketitle

\section{Introduction}
First-passage times of stochastic processes  \cite{Abrahams1986,redner} are of interest for the following reasons, among others. On the theoretical side they represent  non-trivial path functionals of the process, yielding insight into how it explores its state space. On the applied side their study is of merit because, first, many phenomena in nature (including society) are subject to the laws of chance, or can at least be successfully modeled as  being such, and because, second, relevant developments pertaining to these phenomena are often triggered by threshold events, when a certain dynamic quantity first attains a given value or a value in some given set. Examples of this  are legion: options on underlying assets coming in-the-money or out-of-the-money, water levels exceeding crticial flood levels, emergence of queues or of traffic jams, electrical circuit overloads etc.

A basic, but nevertheless often non-trivial task in this area is to identify the laws of the first passage times of a real-valued process below a given level (or above, but by taking the negative of the process we reduce to the former case). Restricting our discussion to Markov processes, in case the sample paths are continuous, i.e. for diffusions, the Laplace transforms, and even the probability density functions  of such first passage times can be expressed to some level of explicitness \cite{nobile}.  For processes with jumps, however, this is in general a very difficult problem, even for those classes of Markov processes for which there is an otherwise well-developed theory. In the present paper we shall  provide some fairly general results concerning this first passage theory for the subclass of Markov processes that have no negative jumps,  a property that simplifies the analysis considerably (and does so more generally for the entire corpus of related ``exit problems'', see e.g.  \cite{kkr,vidmar2021exit,landriault_li_zhang_2017,avram_li_li_2021}).

But our main focus will be in flipping the perspective around and asking, given the nature of the laws of the first passage times, as above, to what extent can we infer the statistics of the underlying Markov process? We shall see that the a.s. absence of negative jumps is indeed embedded already in the structure of these laws, more precisely, of their Laplace transforms.  Moreover, leaning on the general results indicated in the preceding paragraph we shall show that, somewhat surprisingly,  the classes of the processes mentioned in the abstract can be identified from the form of the Laplace transforms of the first passage times, each member of these classes being further uniquely determined by the precise family of laws of its first passage times downwards. 

\subsection{Setting}\label{subsection:setting}
Before we go on to describe what we would like to achieve in more detail  let us fence off precisely the confines within which we shall conduct ourselves.

Throughout $X=(X_t)_{t\in [0,\zeta)}$, under the complete probabilities $(\PP_x)_{x\in I}$,  in the right-continuous filtration $\FF=(\FF_t)_{t\in [0,\infty)}$, is an adapted (implicit: $\zeta$ is an $\FF$-stopping time),  quasi left-continuous, normal and strong Markov process with lifetime $\zeta$, cemetery $\infty$ (accordingly we set $X_t:=\infty$ for $t\in [\zeta,\infty)$), state space an interval $I$ of $\mathbb{R}$ that is unbounded above, c\`adl\`ag in $I$ on $[0,\zeta)$ (by convention $X_{0-}:=X_0$ on $\{\zeta>0\}$), $\lim_{\zeta-}X$ existing in $I\cup\{\infty\}$ a.s. on $\{0<\zeta<\infty\}$ (for definiteness put $\lim_{\zeta-}X:=\infty$ on $\{0<\zeta<\infty\}$ when the limit does not exist), $\zeta>0$ a.s.. We ask that $(I\ni x\mapsto \PP_x(t<\zeta,X_t\in A))$ is at least universally if not Borel measurable for all $A\in \mathcal{B}_I$ ($:=$ the Borel sets of $I$). For any a.s. nondecreasing sequence $(T_n)_{n\in \mathbb{N}}$ of $\FF$-stopping times the event $\{T_n<\zeta\text{ for all }n\in \mathbb{N}\}\cap \{\lim_{n\to\infty}T_n =\zeta\}\cap \{\zeta<\infty,X_{\zeta-}<\infty\}$ is assumed to be negligible: $\zeta$ cannot be announced (is totally inaccessible) on the event that the left limit of $X$ at $\zeta$ is finite.  Finally, to ensure absolute peace of mind --- to be precise, to secure that the first passage times $T_\lll^-$, $\lll\in I$, considered below will be stopping times --- the filtration $\FF$ should enjoy the following property of completeness relative to the family of measures $(\PP_x)_{x\in I}$:  $\FF=\overline{\FF}$, where for $t\in [0,\infty)$ the completed $\sigma$-field $\overline{\FF}_t$ consists of those subsets of the sample space which differ from an event from $\FF_t$ by a negligible set (it is  a relatively weak form of completeness, but it is all we need).
\vspace{0.1cm}
\footnotesize
\begin{adjustwidth}{1cm}{1cm}
Just to be sure, let us fill in some outstanding details. The qualifier a.s. (resp. negligible) meant (and will continue to mean) a.s.-$\PP_x$ (resp. $\PP_x$-negligible) for all $x\in I$, likewise independent will mean independent under each $\PP_x$, $x\in I$, etc. The domains of the probabilities $\PP_x$, $x\in I$, need not (and typically will not) be the same, however they all contain $\FF_\infty$.   By strong Markov we intended that for any $\FF$-stopping time $S$,\footnote{Expectations we write as follows: $\QQ[W]$ for $\mathbb{E}_\QQ[W]$, $\QQ[W;A]$ for $\mathbb{E}_\QQ[W\mathbbm{1}_A]$ and $\QQ[W\vert \mathcal{H}]$ for $\mathbb{E}_\QQ[W\vert \mathcal{H}]$.} $\PP_x[H(X_{S+\cdot})\vert \FF_S]=\PP_{X_S}[H(X)]$ a.s.-$\PP_x$ on $\{S<\zeta\}$ for all $x\in I$ and for all measurable $H\geq 0$; by quasi left-continuity, that for any sequence $(T_n)_{n\in \mathbb{N}}$ of $\FF$-stopping times that is nondecreasing to a $T$ a.s., one has $\lim_{n\to\infty}X_{T_n}=X_T$ a.s. on $\{T<\zeta\}$; finally, by normality, that $X_0=x$ a.s.-$\PP_x$ on $\{\zeta>0\}$ for all $x\in I$. 
\end{adjustwidth}
\vspace{0.1cm}
 \normalsize 
 
 For parts of this paper  some of the properties that we have insisted on above could perhaps be relaxed, but we feel that they are fairly innocuous anyway, so we are happy to leave them as blanket assumptions. Importantly, all the classes of processes mentioned in the abstract will be seen to fall naturally under our general provisons. As yet we have made and make no assumption on the nature of the jumps of $X$. Some further comments on the setting follow. The reader eager to get to the heart of the matter should, and can safely proceed at once to Subsection~\ref{subsection:first-observations-mandate}. 

 \begin{remark}\label{remark:right-cts-filtration}
 Concerning the assumptions made that pertain to the filtration $\FF$ we may mention that if, ceteris paribus, the process $X$ satisfies them instead relative to a right-continuous (but not necessarily complete) filtration $\FF$, then it satisfies them also relative to $\overline{\FF}$ (which is complete and right-continuous), a fortiori relative to the completion $\overline{\FF^X_+}$ of the right-continuous modification $\FF^X_+$ of its natural filtration $\FF^X$. This may be gleaned from the two facts that, for each $x\in I$, (I) $\overline{\FF}$ is included in the usual $\PP_x$-completion $\overline{\FF}^{x}$ of $\FF$ (so in the latter all the $\PP_x$-negligible sets are thrown into each member of the filtration) and (II) for any $\overline{\FF}^x$-stopping time $\overline{S}^x$ there is an $\FF$-stopping time $S$ such that $\overline{S}^x=S$ a.s.-$\PP_x$ and $\overline{\FF}^x_{\overline{S}^x}=\overline{\FF}^x_S=\FF_{S}\lor \PP_x^{-1}(\{0,1\})$ \cite[Theorem~120-IV.59]{dellacherie}. A canonical realization of a Feller process (semigroup) in $I$, which a.s. has a limit in $I\cup\{\infty\}$ at $\zeta-$ on $\{0<\zeta<\infty\}$ (the latter being automatic if $\inf I\in I$), certainly satisfies all our standing assumptions in its right-continuous and completed natural filtration \cite[Sections~III.2 \&~III.3]{revuz-yor}.
\end{remark}

 The requirement on the unboundedness of $I$ above is not an essential issue. For, if, ceteris paribus, $\infty>\sup I\notin I$ and $X$ satisfies the above with $\sup I$ replacing the cemetery state $\infty$ in the considerations of $\lim_{\zeta-}X$ in the obvious way, then we reduce to the standing case through an injective continuous increasing transformation of space, which maps $I$ onto $I\cup (\inf I,\infty)$. On the other hand,  
 if, ceteris paribus, $I\ni \sup I<\infty$, and $X$ satisfies the above, then a reduction is achieved through an enlargement of the state space from to  $I\cup [\sup I,\infty)$, with the new process e.g. having a deterministic strictly negative drift at levels from $[\sup I,\infty)$. Further, if, again ceteris paribus, in lieu of `` $\lim_{\zeta-}X$ existing in $I\cup\{\infty\}$ a.s. on $\{0<\zeta<\infty\}$'' one has just ``$\lim_{\zeta-}X\in I\cup\{\infty,\inf I\}$ a.s. on $\{0<\zeta<\infty\}$, $-\infty<\inf I\notin I$'', then one may attempt to prolongate the process $X$ after its lifetime with the value $\inf I$ on $\{0<\zeta<\infty,\lim_{\zeta-}X=\inf I\}$, adjoining $\inf I$ to the state space as an absorbing state (and if also $\inf I=-\infty$, then first make it finite by a transformation of space, mapping $\mathbb{R}$ to $(l,\infty)$ for some [arbitrary] $l\in\mathbb{R}$).  \label{unbdd-above}
 
 Another possible reduction to keep in mind in applications, when $X$ has no negative jumps, is stopping the process $X$ on exit from $[\lll,\infty)$, restricting the state space to $[\lll,\infty)$, $\lll\in I$ (and then letting $\lll\downarrow \inf I$). 


\subsection{First observations, related literature and mandate}\label{subsection:first-observations-mandate} With the above specification of the setting out of the way let, for $\lll\in I$,  $$T_\lll^-:=\inf \{t\in [0,\zeta):X_t\leq \lll\}\quad (\inf\emptyset:=\infty)$$ be the first passage time of $X$ below the level $\lll$. The $T^-_\lll$, $\lll\in I$, are indeed $\FF$-stopping times. We argue it briefly, to see where and how the right-continuity and completeness of $\FF$ come into play, introducing some notation along the way.
\begin{proof} 
By right-continuity of $\FF$, for all $\lll\in I^\circ:=I\backslash \{\inf I\}$, $T^{\,\text{-}\,\text{-}}_\lll:=\inf\{t\in[0,\zeta):X_t<\lll\}$ is an $\FF$-stopping time, since $\{T^{\,\text{-}\,\text{-}}_\lll<t\}=\cup_{s\in \mathbb{Q}\cap [0,t)}\{s<\zeta,X_s<\lll\}\in \FF_t$ for all $t\in (0,\infty)$. Then fix $\lll\in I$. By quasi left-continuity of $X$ a.s. $T_\lll^-=\tilde T_\lll^-:=\inf \{t\in [0,\zeta):X_t\land X_{t-}\leq \lll\}$, indeed $T_\lll^-\leq \lim_{a\downarrow \lll}T_a^{\,\text{-}\,\text{-}}\leq \tilde T_\lll^-<\zeta$ a.s. on $\{\tilde T_\lll^-<T_\lll^-\}$. But, for $t\in[0,\infty)$, $\{\tilde T_\lll^-\leq t\}= \cup_{s\in (\mathbb{Q}\cap [0,t))\cup \{t\}}\{s<\zeta,\underline{X}_s\leq \lll\}\in \FF_t$, $\underline{X}=(\underline{X}_t)_{t\in [0,\zeta)}$ being the running infimum process of $X$; therefore $\tilde T_\lll^-$  is an $\FF$-stopping time. By completeness of $\FF$ so is $T_\lll^-$.\end{proof} \noindent It stands to reason that the times $T_\lll^-$, $\lll\in I$, are the most natural ones from an applied and intuitive point of view, not their closely allied (and by the hypotheses a.s. equal) counterparts $\tilde T_\lll^-$, $\lll\in I$ (which would have been stopping times without any assumptions on $\FF$), nor indeed the $T_\lll^{\,\text{-}\,\text{-}}$, $\lll\in I^\circ$ (for which only right-continuity of $\FF$ would have been needed). Quasi left-continuity of $X$ and the fact that $\zeta$ cannot be announced on $\{0<\zeta<\infty,X_{\zeta-}<\infty\}$ render
\begin{equation}\label{eq:right-ct}
\text{$\uparrow$-$\lim$}_{b\downarrow a}T_b^-=T_a^-\text{ a.s. for all }a\in I,
\end{equation}
i.e. the  a.s. right-continuity of $T_\lll^-$ in $\lll\in I$ at fixed levels, which will be needed later on.

\label{page:additive-regen}Suppose now (for the time being, for this paragraph) that $X$ a.s. does not have negative jumps. By the strong Markov property of $X$ and the preceding assumption we deduce that the space-indexed process $T^-=(T_\lll^-)_{\lll\in I}$ 
\begin{itemize}
\item is additive, in the sense that for $\lll\leq x$ from $I$, on $\{T_x^-<\infty\}$, the increment $T_\lll^--T_x^-$ is independent of the ``past'' $(T_a^-)_{a\in [x,\infty)}$; 
\item has the regenerative property,  in that for all $a\in [x,\infty)$ for which $\PP_a(T_x^-<\infty)>0$, conditionally on $\{T_x^-<\infty\}$, the law of $T_\lll^--T_x^-$ under $\PP_a$ is $(T_\lll^-)_\star {\PP_x}$.\footnote{$W_\star\mathbb{Q}$ is the push-forward (the law) of a random element $W$ under a probability  $\mathbb{Q}$.}
\end{itemize}
In particular the law of the process $T^-$ is determined already by its one-dimensional marginals. Furthermore, when
\begin{equation}\label{eq:positive-chances}
\text{$\PP_x(T_\lll^-<\zeta)>0$ for all $\lll\leq x$ from $I$},
\end{equation}
we may infer from this, for each $q\in [0,\infty)$, the existence, unique up to a multiplicative constant,  of a so-called scale function $\Phi_q:I\to (0,\infty)$ such that 
\begin{equation}\tag{$q$}\label{eq:scale-identity}
\PP_x[e^{-q T_\lll^-};T_\lll^-<\zeta]=\frac{\Phi_q(x)}{\Phi_q(\lll)}\text{ for all $\lll\leq x$ from $I$.}
\end{equation}
\begin{proof}
We argue existence, uniqueness in the sense indicated is clear.  For  $\lll\leq x\leq a$ from $I$, 
\begin{align*}
\PP_a[e^{-q T_\lll^-};T_\lll^-<\zeta]&=\PP_a[e^{-q( T_x^-+(T_\lll^--T_x^-))};T_\lll^-<\infty]\\
&=\PP_a[e^{-q T_x^-};T_x^-<\infty]\PP_a[e^{-q(T_\lll^--T_x^-)}\mathbbm{1}_{\{T_\lll^-<\infty\}}\vert T_x^-<\infty]\quad \text{(by additivity of $T^-$)}\\
&=\PP_a[e^{-q T_x^-};T_x^-<\infty]\PP_x[e^{-qT_\lll^-}; T_\lll^-<\infty]\text{ (by the regenerative property of $T^-$)}.
\end{align*}
Thus 
$$\PP_x[e^{-qT_\lll^-}; T_\lll^-<\zeta]=\frac{\PP_a[e^{-q T_\lll^-};T_\lll^-<\zeta]}{\PP_a[e^{-q T_x^-};T_x^-<\zeta]}\text{ (on using \eqref{eq:positive-chances})}.$$ Choosing finally a sequence $(a_k)_{k\in \mathbb{N}}$ in $I$ that is increasing to $\infty$ we define unambiguously $\Phi_q(z):=\frac{(\PP_{a_k}[e^{-q T_z^-};T_z^-<\zeta])^{-1}}{(\PP_{a_k}[e^{-q T_{a_1}^-};T_{a_1}^-<\zeta])^{-1}}$ for $z\in I\cap (-\infty,a_k]$, $k\in \mathbb{N}$, and get \eqref{eq:scale-identity}.  
\end{proof}

It has emerged in the literature \cite{kyprianou,ma,vidmar2021continuousstate,pierre} that for many classes of processes (satisfying our standing assumptions, of course) the scale functions $\Phi_q$, $q\in [0,\infty)$, may be identified in explicit form in terms of the characteristics of the process $X$. (An entirely analogous, technically less involved, phenomenon can be wittnessed in discrete space, see e.g. \cite{vidmar-branch}, but we do not deal with the latter here.)

Since the structure \eqref{eq:scale-identity} is somehow responsible for a considerable simplification of the first passage problem we find it worthwhile first to  characterize the situation in which \eqref{eq:scale-identity} can occur at all.  

Once this has been achieved our interest will lie  in the (further) investigation of the interplay between the law of $X$, the law of $T^-$, and the scale functions $\Phi_q$, $q\in [0,\infty)$, of \eqref{eq:scale-identity}. Specifically, on the one hand, we will want to characterize the maps $\Phi_q$, $q\in[0,\infty)$, as precisely as we can given the relative generality of the setting. Conversely, holding $I$ fixed, our goal will be to show that the (resp. classes of) laws of the processes mentioned in the abstract are already uniquely determined by (a) (resp. certain forms of)  the functions $\Phi_q$ satisfying \eqref{eq:scale-identity} for $q\in [0,\infty)$ in a neighborhood of infinity and  (b) the ``boundary behaviour'' stipulation that $\inf I$ is absorbing if it belongs to $I$; in other words, no other Markov processes, subject to the assumptions of Subsection~\ref{subsection:setting} and having $\inf I$ absorbing when $\inf I\in I$, can mimick the laws of their first passage times (resp. in the form of their Laplace transforms).

The former problem is in some sense basic enough. The process $T^-$ is a generic (not depending on the law of $X$)  map of $X$, that is measurable relative to the completion \cite[Eq.~(I.5.2)]{ge2011markov}  w.r.t. the family of laws $X_\star \PP_x$, $x\in I$ (whatever these laws, subject to the provisions of Subsection~\ref{subsection:setting}); therefore, at least in principle, it must be possible to describe/characterize its law, i.e. the $\Phi_q$, $q\in [0,\infty)$, in relatively explicit terms. 

The resolution to the second problem appears rather a bit more involved. Indeed, in order for it to be trivial we would presumably need, at the very least, for all fixed $x\in I$ and $t\in [0,\infty)$, $X_t$ to be a.s.-$\PP_x$ equal to a generic (depending on $t$ and maybe even on $x$, but not on the law of $X$) map of $T^-$, that would be $(T^-)_\star\PP_x$-measurable (again, whatever this law, subject to the provisions of Subsection~\ref{subsection:setting} and $\inf I$ being absorbing when it belongs to $I$). This, however is not the case, which seems plain enough, but it would perhaps also seem obvious that Brownian motion could not possibly be a.s. recovered  measurably from its drawdown process, yet it can. So let us be explicit about it for $I=\mathbb{R}$. If, per absurdum, it were possible, then a Brownian motion with strictly positive drift would a.s.-$\PP_0$ be equal to a measurable function of its running infimum. However, the post-(time of overall infimum) increments of such drifted Brownian motion are independent of the running infimum and non-trivial under $\PP_0$, so it cannot be. 

Thus, it is not at all obvious why the scale functions $\Phi_q$, $q\in [0,\infty)$, together with the ``boundary behaviour'' stipulation, should determine already the law of $X$. The only result in this direction that the author is aware of is a recent finding of \cite{mateusz}, generalizing \cite{doney-chaumont}, namely that \emph{within the L\'evy class} a process is determined by the distribution of its running minimum (whereas we shall want to determine (classes of) processes in terms of the distributions of their first passage times downwards \emph{subject only to the assumptions of Subsection~\ref{subsection:setting}}). More broadly, and somewhat related, characterizations of the class  of first passage time distributions for classes of Markov processes have been a subject of interest, see e.g. \cite{bondesson} for random walks and \cite{makoto} for birth and death processes. On the other hand, focusing instead on the distributions of the \emph{position} of a Markov process on hitting a set (hitting distributions), it was shown in the seminal paper \cite{b-g-m} that, loosely speaking, two Hunt processes having identical hitting distributions are equal in law up to a time-change; a little later it was  established \cite{arbib} that a sample-path continuous process having, in a sense that we do not make precise here,  the hitting probabilities and mean exit times of a given diffusion,  is Markovian and has the transition probabilities of the diffusion (so here one does not even need to assume Markovianity of the process).

\subsection{Article structure and roadmap to results}\label{subsection:structure}
In Section~\ref{sec:scale-at-the-min} we characterize processes for which there ``exists a scale function $\Phi_q$ of  \eqref{eq:scale-identity}, $q\in (0,\infty)$, at the minimum'' (Theorem~\ref{theorem:skip-free}). It turns out, without surprise, that actually it is necessary for $X$ to a.s. not have negative jumps. 

 Section~\ref{section:grround-work} describes the $\Phi_q$   of \eqref{eq:scale-identity}, $q\in [0,\infty)$, in terms of a martingale problem (Proposition~\ref{proposition:characterization-one}) and via the associated generator-eigenvalue equation  (Proposition~\ref{proposition:characterization-generator}); said generator is identified explicitly, and some special properties of the scale functions are noted in the ``time-changed L\'evy world'' (Proposition~\ref{proposition:fund-class}). We do not claim any novelty here (all mentioned propositions are folklore, for sure), only relative generality, which may serve as a useful reference.

Exploiting the preceding, Theorem~\ref{theorem:skip-free} and  Proposition~\ref{proposition:characterization-one} especially, characterizations of the law of $X$ through the laws of its first passage times then follow in our main series of results in Section~\ref{section:characterizations}. Specifically, for possibly killed spectrally positive L\'evy processes (pk-spLp), see Theorem~\ref{theorem:levy-character}; for  self-similar Markov processes on the real line of the spectrally positive type, Theorem~\ref{theorem:self-similar}; for continuous-state branching processes,  Theorems~\ref{theorem:csbp} and~\ref{theorem:csbp-2}. These results are at least a little bit unexpected. For intuitively it seems that the times of first passage downwards at the very least ``could not see'' the positive jumps, and yet, in the sense of characterizing the law, they do. 


\subsection{Convention}
When referring to a (possibly killed) spectrally positive L\'evy process we exclude (possibly killed) subordinators, in particular the constant process, but \emph{not} negative drifts. This is relatively very natural given that our analysis is  of the first passage times downwards, these being trivial for (a.s.) nondecreasing processes.

\section{Characterization of processes that admit a scale function at the minimum}\label{sec:scale-at-the-min}
The reader would have noticed that in deriving \eqref{eq:scale-identity}, besides \eqref{eq:positive-chances}, only the following property, not the a.s. absence of the negative jumps, of $X$ was actually used (in applying the strong Markov property to get the additive and regenerative properties of $T^-$). 
\begin{definition}
$X$ is said to be continuous at the minimum when  $X_{T_a^-}=a$ a.s.-$\PP_x$ on $\{T_a^-<\zeta\}$ for all $a\leq x$ (equivalently, all $a\in (\inf I, x]$, all $x$) from $I$.
\end{definition}
In fact the two properties are equivalent, and (much) more can be said.

\begin{theorem}\label{theorem:skip-free}
The following statements are equiveridical. 
\begin{enumerate}[(i)]
\item\label{scales-equiv:i} $X$ a.s. does not have negative jumps and \eqref{eq:positive-chances} holds true.
\item\label{scales-equiv:i'} $\underline{X}$ is a.s. continuous and \eqref{eq:positive-chances} holds true.
\item\label{scales-equiv:iii}  $X$ is continuous at the minimum and \eqref{eq:positive-chances} holds true.
\item\label{scales-equiv:iii'} $T^-$ is additive and regenerative  and \eqref{eq:positive-chances} holds true.
\item\label{scales-equiv:ii'} For each $q\in [0,\infty)$ there exists a map $\Phi_q:I\to (0,\infty)$ such that \eqref{eq:scale-identity} holds true. 
\item\label{scales-equiv:ii} For some $q\in (0,\infty)$ there exists a map $\Phi_q:I\to (0,\infty)$ such that \eqref{eq:scale-identity} holds true. 
\end{enumerate}
The equivalence of \ref{scales-equiv:i}, \ref{scales-equiv:i'} and \ref{scales-equiv:iii} holds true even if \eqref{eq:positive-chances} is omitted from them. Furthermore, when the equivalent conditions above hold true, then, for each $q\in [0,\infty)$, the map $\Phi_q$ is continuous and nonincreasing; if $q>0$, then even strictly decreasing.
\end{theorem}
Besides its theoretical appeal, this result will lend itself nicely in proving the characterizations of Section~\ref{section:characterizations}, since it allows to infer an important structural property of the process (namely, the a.s. absence of negative jumps) directly from the form of the Laplace transforms of the first passage times. Before turning to the proof it is perhaps worth remarking that additivity (in particular, Markovianity) of $T^-$ together with \eqref{eq:positive-chances} is not enough to infer the absence of negative jumps:
\begin{example}
Let $X$ be the negative of a compound Poisson subordinator with exponentially distributed  jump sizes. By the memoryless property of the exponential distribution it follows that the process $T^-$ is additive. By Theorem~\ref{theorem:skip-free} for no $q\in (0,\infty)$ can we have \eqref{eq:scale-identity}. As a check we can compute (e.g. by conditioning on the first jump and solving the resulting integral equation), $\lambda\in (0,\infty)$ being the holding rate and $\mu\in (0,\infty)$ the rate of the jump sizes of $X$,
$$\PP_x[e^{-q T_\lll^-};T_\lll^-<\zeta]=
\begin{cases}
1, &\lll=x\\
\frac{\lambda}{q+\lambda}e^{-\frac{\mu q}{q+\lambda}(x-\lll)},&\lll<x
\end{cases},\quad \{\lll, x\}\subset\mathbb{R},$$
and it is clear that this cannot be written in the form \eqref{eq:scale-identity}.
\end{example}

\begin{proof}
We have already noted in Subsection~\ref{subsection:first-observations-mandate} and at the start of this section that \ref{scales-equiv:iii}   implies \ref{scales-equiv:iii'}, which in turn was seen to entail \ref{scales-equiv:ii'}, and trivially 
\ref{scales-equiv:ii'} implies \ref{scales-equiv:ii}.

Assume \ref{scales-equiv:ii}. Let $a\in (\inf I,x]$ and pick any $\lll\in I$ such that $\lll< a$. Then, on the one hand, directly from  \eqref{eq:scale-identity},
$$\PP_x[e^{-q T_\lll^-};T_\lll^-<\zeta]=\frac{\Phi_q(x)}{\Phi_q(\lll)}=\frac{\Phi_q(x)}{\Phi_q(a)}\frac{\Phi_q(a)}{\Phi_q(\lll)}=\PP_x[e^{-q T_a^-}\PP_a[e^{-qT_\lll^-};T_\lll^-<\zeta];T_a^-<\zeta]$$ and on the other hand, by the strong Markov property,
$$\PP_x[e^{-q T_\lll^-};T_\lll^-<\zeta]=\PP_x\left[e^{-q T_a^-}\PP_{X_{T_a^-}}[e^{-qT_\lll^-};T_\lll^-<\zeta];T_a^-<\zeta\right].$$
The fact that $\Phi_q$ is nonincreasing (which is immediate from \eqref{eq:scale-identity}, since the Laplace transform is bounded by $1$) together with \eqref{eq:scale-identity} and $X_{T_a^-}\leq a$ on $\{T_a^-<\zeta\}$ ensures that 
$$\PP_{X_{T_a^-}}[e^{-qT_\lll^-};T_\lll^-<\zeta]\geq \PP_a[e^{-qT_\lll^-};T_\lll^-<\zeta]\text{ on }\{T_a^-<\zeta\}.$$ We conclude that actually 
$$\PP_{X_{T_a^-}}[e^{-qT_\lll^-};T_\lll^-<\zeta]= \PP_a[e^{-qT_\lll^-};T_\lll^-<\zeta]\text{ a.s.-}\PP_x\text{ on }\{T_a^-<\zeta\}.$$ Because $\Phi_q$ is even strictly decreasing (the contrary would imply by \eqref{eq:scale-identity} that for some $\lll<x$ from $I$ we would have $T_\lll^-=0$ a.s.-$\PP_x$, which is absurd by the right-continuity of $X$) and $\lll<a$ it follows via \eqref{eq:scale-identity} again that $X_{T_a^-}=a$ a.s.-$\PP_x$ on $\{T_a^-<\zeta\}$, so we get \ref{scales-equiv:iii} (it is elementary to note from \eqref{eq:scale-identity} that \eqref{eq:positive-chances} holds true).

Clearly the a.s. absence of negative jumps implies that $\underline{X}$ is continuous, which in turn implies continuity at the minimum. Suppose conversely that $X$ is continuous at the minimum and, per absurdum, that there is an $x\in I$ such that with positive $\PP_x$-probability $X$ has negative jumps. By continuity of probability, for some $\epsilon\in (0,\infty)$, $X$ has a negative jump of size $>\epsilon$ with positive $\PP_x$-probability. For $t\in [0,\zeta)$, denote by $S_t$  the time of the first negative jump of $X$ of size $>\epsilon$ on the temporal interval $(t,\zeta)$ (which may be infinite, though $\PP_x(0<\zeta,S_0<\zeta)>0$).
For $n\in \mathbb{N}_0$ the event 
$$\cup_{k\in \mathbb{N}_0}\left\{\frac{k}{2^n}<\zeta,S_{\frac{k}{2^n}}<\zeta,\sup_{t\in [0,S_{k/2^n})} \vert X_{\frac{k}{2^n}+t}-X_{\frac{k}{2^n}}\vert \leq \frac{\epsilon}{2}\right\}$$
is $\uparrow \{0<\zeta,S_0<\zeta\}$ as $n\to \infty$ (because $X$ is l\`ag on $(0,\zeta)$). Hence,  by continuity of probability again, and then by superadditivity, for some $n\in \mathbb{N}_0$ and some $k\in \mathbb{N}_0$,  the event 
$$\left\{\frac{k}{2^n}<\zeta,S_{\frac{k}{2^n}}<\zeta,\sup_{t\in [0,S_{k/2^n})} \vert X_{\frac{k}{2^n}+t}-X_{\frac{k}{2^n}}\vert \leq \frac{\epsilon}{2}\right\}$$ has positive $\PP_x$-probability. By the Markov property at time $\frac{k}{2^n}$ we get that with positive $(X_{\frac{k}{2^n} })_\star{\PP_x}\vert_{\{\frac{k}{2^n}<\zeta\}}$-measure in $z\in I$, therefore certainly for some $z\in I$, $$\PP_z(X\text{ attains a new minimum by jumping strictly over its past running minimum})>0.$$ Finally, the latter implies that for some (say) rational $a\in I$, $a<z$, with positive $\PP_z$-probability on $\{T_a^-<\zeta\}$, $X_{T_a^-}<a$, which is a contradiction.

Thus the equivalence of \ref{scales-equiv:i}-\ref{scales-equiv:ii}, and also of  \ref{scales-equiv:i}-\ref{scales-equiv:iii} absent \eqref{eq:positive-chances}, is established. 

Assume finally that $X$ meets the equivalent conditions \ref{scales-equiv:i}-\ref{scales-equiv:ii}. Let $q\in [0,\infty)$. It is plain from \eqref{eq:scale-identity} that $\Phi_q$ is nonincreasing and even strictly decreasing if $q>0$.  Besides, thanks to \eqref{eq:right-ct}, for $q>0$, by bounded convergence in \eqref{eq:scale-identity} in $\lll$ at fixed $x\in (\lll,\infty)$, $\Phi_q$ must be right-continuous at any $\lll\in I$. For $q=0$ the argument for the right-continuity of $\Phi_q$ is a little bit more delicate, because in principle it could happen that with positive $\PP_x$-probability $\zeta>T_b^-\uparrow \infty$ as $b\downarrow \lll$. Actually this is not possible in virtue of Kolmogorov's zero-one law: conditionally on $\cap_{b\in (\lll,\infty)}\{T_b^-<\zeta\}$, the sequence $(T_{\lll+\frac{1}{n+1}}^--T_{\lll+\frac{1}{n}}^-)_{n\in \mathbb{N}}$ 
 is a $\PP_x$-independency having $\{\lim_{b\downarrow \lll}T_b^-=\infty\}$ as a tail event; the event is not conditionally $\PP_x$-almost certain due to \eqref{eq:positive-chances}, therefore must be conditionally $\PP_x$-negligible, i.e. (by \eqref{eq:right-ct}) $T_\lll^-<\zeta$ a.s.-$\PP_x$ on $\cap_{b\in (\lll,\infty)}\{T_b^-<\zeta\}$. Hence even for $q=0$ one can pass to the limit in  \eqref{eq:scale-identity} by continuity of probability from above. 
As for left-continuity, suppose $\Phi_q$ were not left-continuous at some $x\in I^\circ$. Letting $\lll\uparrow x$ in \eqref{eq:scale-identity} it would mean by monotone convergence that with positive $\PP_x$-probability we would have $T_x^{\,\text{-}\,\text{-}}=\lim_{\lll\uparrow x}T_\lll^->0$. At the same time, by \eqref{eq:positive-chances} we have $\PP_x(T_x^{\,\text{-}\,\text{-}}<\zeta)>0$. Altogether this contradicts the strong Markov property of $X$ at the time $T_x^{\,\text{-}\,\text{-}}$ (taking into account that $X_{T_x^{\,\text{-}\,\text{-}}}=x$ a.s.-$\PP_x$ on $\{T_x^{\,\text{-}\,\text{-}}<\zeta\}$).
\end{proof}
It is quite agreeable that, when the scale functions of  \eqref{eq:scale-identity} exist at all, then automatically they have to be continuous, which entails all of the following:
\begin{itemize}
\item (by left-continuity of $\Phi_q$ for any $q\in [0,\infty)$) $X$ must be regular downwards, i.e. it must enter $(-\infty,x)\cap I$ immediately a.s.-$\PP_x$ for all $x\in I^\circ$; 
\item consequently (by the strong Markov property) $$T_\lll^{\,\text{-}\,\text{-}}=\text{$\downarrow$-$\lim$}_{b\uparrow \lll}T_b^-=T_\lll^-\text{ a.s. for all }\lll\in I^\circ,$$ so that  $T^-_\lll$ is a.s. left-continuous in $\lll\in I^\circ$ at fixed levels, hence together with \eqref{eq:right-ct} continuous at fixed levels a.s. (but usually not continuous a.s., of course);
 \item (by right-continuity of $\Phi_0$) $T_\lll^-<\zeta$ a.s. on $\cap_{b\in (\lll,\infty)}\{T_b^-<\zeta\}$ for all $\lll\in I$.
 \end{itemize}

\section{The scale functions} \label{section:grround-work}
\begin{center}
\fbox{\parbox{0.56\textwidth}{
In this section we assume that $X$ a.s. has no negative jumps.
}}
\end{center}
\subsection{Martingale description}
For $q\in [0,\infty)$ the scale function $\Phi_q$ of \eqref{eq:scale-identity} is basically the solution to a martingale problem.


%

\begin{proposition}\label{proposition:characterization-one}
Let $q\in [0,\infty)$ and let $\Phi:I\to \mathbb{R}$ be Borel (or even just universally measurable) and  bounded on $[\lll,\infty)$ for all $\lll\in I$. Then the  following two statements are equivalent.
\begin{enumerate}[(i)]
\item \label{equivalent:iii}  For all $\lll\in I$ the process $(\Phi(X_{t\land T_\lll^-})e^{-q(t\land T_\lll^-)}\mathbbm{1}_{\{t\land T_\lll^-<\zeta\}})_{t\in [0,\infty)}$ is a martingale  with $\PP_x$-terminal value $\Phi(\lll)e^{-qT_\lll^-}\mathbbm{1}_{\{T_\lll^-<\zeta\}}$ a.s.-$\PP_x$ for all $x\in [\lll,\infty)$.
\item \label{equivalent:vi} $\Phi(\lll)\PP_x[e^{-qT_\lll^-};T_\lll^-<\zeta]=\Phi(x)$ for all $x\in [\lll,\infty)$, all $\lll\in I$.
\end{enumerate}
Besides, if $\Phi$ is continuous and $q>0$, then  \ref{equivalent:iii}-\ref{equivalent:vi} are further equivalent to the next statement.
\begin{enumerate}[(i)]
 \setcounter{enumi}{2}
 \item\label{equivalent:i} If $\inf I\in I$, then $(\Phi(X_{t\land T_{\inf I}^-})e^{-q(t\land T_{\inf I}^-)}\mathbbm{1}_{\{t\land T_{\inf I}^-<\zeta\}})_{t\in [0,\infty)}$ is a martingale. If $\inf I\notin I$, then $(\Phi(X_{t})e^{-qt}\mathbbm{1}_{\{t<\zeta\}})_{t\in [0,\infty)}$ is a local martingale.
\end{enumerate}
Even if $q=0$ or $\Phi$ is not continuous, \ref{equivalent:iii}-\ref{equivalent:vi} imply \ref{equivalent:i}.
\end{proposition}
\begin{remark}\label{remark:q>0}
When $q>0$ the terminal value condition in \ref{equivalent:iii} can be dropped (it follows automatically). Not for $q=0$ (even just a L\'evy process can have $\zeta=\infty$ so that $\Phi\equiv 1$ gives a martingale in \ref{equivalent:iii}, yet be drifting to $\infty$ so that \ref{equivalent:vi} fails with $\Phi\equiv 1$ [both for $q=0$]). Usually the case $q=0$ is handled by taking the limit $q\downarrow 0$.
\end{remark}
\begin{remark}
If $\Phi>0$ then \ref{equivalent:vi}  implies \eqref{eq:positive-chances} and is equivalent to \eqref{eq:scale-identity} for $\Phi_q=\Phi$. We do not want to assume  \eqref{eq:positive-chances} outright, because usually this is a non-trivial observation to be made about the process, and finding a $\Phi$ that is $>0$ is one way to do it. If \eqref{eq:positive-chances} should be found to fail, then all is not lost in terms of \eqref{eq:scale-identity} for it may be possible to decompose the state space of $X$ into classes within which $X$ ``communicates downwards'' and treat the first passage problem downwards separately within each class.
\end{remark}
\begin{proof}
If \ref{equivalent:vi} holds true, then  \ref{equivalent:iii} does also, which we deduce using the Markov property  for $X$ and the continuity of $X$ at the minimum: 
\begin{align*}
\PP_x[\Phi(\lll)e^{-qT_\lll^-}\mathbbm{1}_{\{T_\lll^-<\zeta\}}\vert \FF_t]&=\Phi(\lll)e^{-qT_\lll^-}\mathbbm{1}_{\{T_\lll^-\leq t\}}+\Phi(\lll)e^{-qt}\PP_{X_t}[e^{-qT_\lll^-};T_\lll^-<\zeta]\mathbbm{1}_{\{t<\zeta\land T_\lll^-\}}+0\cdot \mathbbm{1}_{\{\zeta\leq t<T_\lll^-=\infty\}}\\
&=\Phi(\lll)e^{-qT_\lll^-}\mathbbm{1}_{\{T_\lll^-\leq t\}}+\Phi(X_t)e^{-qt}\mathbbm{1}_{\{t<\zeta\land T_\lll^-\}}\\
&=\Phi({X}_{t\land T_\lll^-})e^{-q(t\land T_\lll^-)}\mathbbm{1}_{\{t\land T_\lll^-<\zeta\}}
\end{align*}
a.s.-$\PP_x$ for all $x\in [\lll,\infty)$, $\lll\in I$ and $t\in [0,\infty)$. From \ref{equivalent:iii} we get \ref{equivalent:vi} by taking conditional expectation of the terminal value at time zero (and by using the normality of $X$ and $\zeta>0$ a.s.). 

 Statements~\ref{equivalent:iii}-\ref{equivalent:vi} imply \ref{equivalent:i} because in case $\inf I\notin I$, $T_\lll^-\uparrow \infty$ a.s. as $\lll\downarrow \inf I$  (which is true thanks to $\lim_{\zeta-}X$ existing in $I\cup \{\infty\}$ a.s. on $\{0<\zeta<\infty\}$ and the l\`ag property of $X$ in $I$ on $(0,\zeta)$). The converse follows by optional stopping (here we use right-continuity of the processes, which we get from the continuity of $\Phi$ and from $X$ being c\`ad) and from  Remark~\ref{remark:q>0} (which uses $q>0$).
\end{proof}

\subsection{Generator description}
The $\Phi_q$, $q\in [0,\infty)$, of \eqref{eq:scale-identity} are also characterized neatly via the (martingale) generator of $X$, at least when $q>0$ (recall Remark~\ref{remark:q>0}). 

In preparation thereof,  let  $f:I\to \mathbb{R}$ be continuous  (we will be doing optional stopping and for this we need some right-continuity of processes), bounded on $[a,\infty)$  for all $a\in I$ and let $\AA f:I^\circ\to \mathbb{R}$ be locally bounded Borel. No separate existence of an object $\AA$ is a priori implied here, and formally $\AA f$ is to be treated as a ``single symbol'' that could just as well be e.g. $g$, but informally $\AA$ will play the role of the generator of $X$; we will comment on this piece of notation further below. Let also $a\in I^\circ$. We make the following observations.

$(\bullet_1)$ For a given $q\in [0,\infty)$,
\begin{align}
&\text{there exists a sequence $T=(T_n)_{n\in \mathbb{N}}$ of $\FF$-stopping times bounded by $\zeta$ that is a.s. $\uparrow \zeta$}\nonumber\\
\nonumber
&\text{and satisfying $T_n<\zeta$ on $\{0<\zeta<\infty,X_{\zeta-}=\infty\}$ for each $n\in\mathbb{N}$}\\\nonumber
&\text{(but [automatically] $T_n=\zeta$  for all but finitely many $n\in \mathbb{N}$ a.s. on $\{\zeta<\infty,X_{\zeta-}<\infty\}$)}\\\nonumber
&\text{and such that}\\
\label{eq:loc-mtg}
&\left(f(X_{t\land T_a^-\land T_n})e^{-q(t\land T_a^-\land T_n)}\mathbbm{1}_{\{t\land T_a^-\land T_n< \zeta\}}-\int_0^{t\land T_a^-\land T_n} (\AA f(X_s)-qf(X_s))e^{-qs}\dd s\right)_{t\in [0,\infty)}\\
&\text{is a local martingale for all $n\in \mathbb{N}$,}\tag{${f,\AA f}\choose {q,a}$}\label{A}
 \end{align}\label{page:discussion}
holds true if and only if the latter is verified for the sequence $T$ having $T_n$ equal to $$T_n^+:=\inf\{t\in [0,\zeta):X_t> n\}\land\zeta,\quad n\in \mathbb{N}$$ (sufficiency is trivial; necessity is seen by stopping \eqref{eq:loc-mtg} at $T_m^+$, $m\in\mathbb{N}$, and passing to the limit $n\to\infty$). 

$(\bullet_2)$ For all $n\in \mathbb{N}$: $(\bullet_{2a})$ the process of \eqref{eq:loc-mtg} with $T_n=T_n^+$  is even a.s. bounded on each bounded deterministic interval (hence a martingale if it is a local martingale), no matter what the $q\in [0,\infty)$; and,  $(\bullet_{2b})$ \eqref{eq:loc-mtg} (still with $T_n=T_n^+$ ) is a martingale for $q=0$ iff it is a martingale with some given $q\in (0,\infty)$, in which case it is a martingale for all $q\in [0,\infty)$. 
\begin{proof}$(\bullet_{2a})$ is immediate. 

To see  $(\bullet_{2b})$, assume (*) \eqref{eq:loc-mtg} is a martingale for some $q\in [0,\infty)$, let $S$ be a bounded $\FF$-stopping time, write $\eta:=T_a^-\land T_n^+\land S$ for short, and compute for $x\in I$, $q\in (0,\infty)$, as follows, noting that all the maps against which we Lebesgue-Stieltjes integrate (in $\dd_t$ below, $t$ indicating the integration ``dummy'' variable) are of bounded variation on account of (*): 
\begin{align*}
& \PP_x\left[f(X_{ \eta})e^{-q\eta}\mathbbm{1}_{\{ \eta< \zeta\}}-\int_0^{ \eta} (\AA f(X_s)-qf(X_s))e^{-qs}\dd s\right]\\
&= \PP_x\left[f(X_{ \eta})e^{-q\eta}\mathbbm{1}_{\{ \eta< \zeta\}}+\int_0^\eta  qe^{-qs}f(X_s)\dd s\right]-\int_0^\infty  e^{-qs}\PP_x[\AA f(X_s);s<\eta]\dd s\\
&= \PP_x\left[f(X_{ \eta})e^{-q\eta}\mathbbm{1}_{\{ \eta< \zeta\}}+\int_0^\eta  qe^{-qs}f(X_s)\dd s\right]-\int_0^\infty  e^{-qt}\dd_t \left(\int_0^t\PP_x[\AA f(X_s);s<\eta]\dd s\right)\\
&= \PP_x\left[f(X_{ \eta})e^{-q\eta}\mathbbm{1}_{\{ \eta< \zeta\}}+\int_0^\eta  qe^{-qs}f(X_s)\dd s\right]-\int_0^\infty  e^{-qt}\dd_t \left(\PP_x\left[\int_0^{t\land \eta}\AA f(X_s)\dd s\right]\right)\\
&= \PP_x\left[f(X_{ \eta})e^{-q\eta}\mathbbm{1}_{\{ \eta< \zeta\}}+\int_0^\eta  qe^{-qs}f(X_s)\dd s\right]-\int_0^\infty  e^{-qt}\dd_t \left(\PP_x[f(X_{t\land \eta});t\land \eta<\zeta]-f(x)\right)\\
&\quad +\int_0^\infty  e^{-qt}\dd_t \left(\PP_x[f(X_{t\land \eta});t\land \eta<\zeta]-f(x)-\PP_x\left[\int_0^{t\land \eta}\AA f(X_s)\dd s\right]\right)\\
&= \PP_x\left[f(X_{ \eta})e^{-q\eta}\mathbbm{1}_{\{ \eta< \zeta\}}+\int_0^\eta  qe^{-qs}f(X_s)\dd s\right]-\int_0^\infty q e^{-qt} \left(\PP_x[f(X_{t\land \eta});t\land \eta<\zeta]-f(x)\right)\dd t\\
&\quad +\int_0^\infty  e^{-qt}\dd_t \left(\PP_x\left[f(X_{t\land \eta})\mathbbm{1}_{\{t\land \eta<\zeta\}}-f(x)-\int_0^{t\land \eta}\AA f(X_s)\dd s\right]\right) \text{ (via per partes)}\\
&=f(x)+\PP_x\left[f(X_{ \eta})e^{-q\eta}\mathbbm{1}_{\{ \eta< \zeta\}}+\int_0^\eta  qe^{-qs}f(X_s)\dd s-\int_0^{\infty} q e^{-qt}f(X_{t\land \eta})\mathbbm{1}_{\{ t\land \eta< \zeta\}}\dd t\right] \\
&\quad +\int_0^\infty  e^{-qt}\dd_t \left(\PP_x\left[f(X_{t\land \eta})\mathbbm{1}_{\{t\land \eta<\zeta\}}-f(x)-\int_0^{t\land \eta}\AA f(X_s)\dd s\right]\right)\\
&=f(x)+\int_0^\infty  e^{-qt}\dd_t \left(\PP_x\left[f(X_{t\land \eta})\mathbbm{1}_{\{t\land \eta<\zeta\}}-f(x)-\int_0^{t\land \eta}\AA f(X_s)\dd s\right]\right),
\end{align*}
since (for the last equality) on the event $\{\eta\geq \zeta\}$, $\eta=\zeta$. 

Now, if \eqref{eq:loc-mtg} is a martingale for $q=0$, then we get from the preceding computation, for a general $q$, that $ \PP_x\left[f(X_{ \eta})e^{-q\eta}\mathbbm{1}_{\{ \eta< \zeta\}}-\int_0^{ \eta} (\AA f(X_s)-qf(X_s))e^{-qs}\dd s\right]=f(x)$. This being true for all bounded $\FF$-stopping times $S$ renders \eqref{eq:loc-mtg} a martingale with this general $q$.  

Conversely, if  \eqref{eq:loc-mtg} is a martingale for some $q\in (0,\infty)$, then we get, again by the preceding computation, with this $q$, that $\int_0^\infty  e^{-qt}\dd_t \left(\PP_x\left[f(X_{t\land \eta})\mathbbm{1}_{\{t\land \eta<\zeta\}}-f(x)-\int_0^{t\land \eta}\AA f(X_s)\dd s\right]\right)=0$. This being true also if we replace $S$ with $S\land a$ for an arbitrary $a\in [0,\infty)$, we see that $\int_0^\cdot  e^{-qt}\dd_t \left(\PP_x\left[f(X_{t\land \eta})\mathbbm{1}_{\{t\land \eta<\zeta\}}-f(x)-\int_0^{t\land \eta}\AA f(X_s)\dd s\right]\right)$ vanishes identically as a function of the upper delimiter, which can only be if $\PP_x\left[f(X_{t\land \eta})\mathbbm{1}_{\{t\land \eta<\zeta\}}-f(x)-\int_0^{t\land \eta}\AA f(X_s)\dd s\right]$ is in fact constant (and hence $=0$). The latter in turn, $S$ still being arbitrary, renders \eqref{eq:loc-mtg}  a martingale for $q=0$.
\end{proof}
  The reader may compare  $(\bullet_{2b})$ with  \cite[Lemma~4.3.2]{ethier} for the case when $\zeta=\infty$  (which is incidental by extension of space with a cemetery) and without the ``localization'' with the $T_a^-\land T_n^+$ (which as we have just seen also does not really matter, but this was not obvious to begin with). \label{page:after-proof}
  
The reason we ``localize'' with $T_n$ and $T_a^-$ in \eqref{eq:loc-mtg} is, informally speaking, because there may be problems with explosions and  because there may also be issues with the behavior at the boundary $\inf I$. It is not just an intellectual exercise -- the author himself had to contend with both of these unpleasantries in \cite{vidmar2021continuousstate}. In any event, the point is that under our standing assumptions it will be enough to work with the  ``local  behavior'' of $X$, in and only up to exit from $I^\circ$. 

In practice, for a given $f$, we have in mind that the existence of an $\AA f$ satisfying \eqref{A} for all $q\in [0,\infty)$ and $a\in I^\circ$ comes  from an It\^o(-type) formula and/or by time-change, and this is really where the strength of working with generators comes from; we return to this on a relatively pleasant  class of processes in Proposition~\ref{proposition:fund-class}. 

As for uniqueness, for a given $f$ and $q\in [0,\infty)$, we may note that at most one continuous $\AA f:I^\circ\to \mathbb{R}$ exists satisfying \eqref{A} for all $a\in I^\circ$. Indeed, if $\tilde\AA f$ also meets these requirements, then we get that $\int_0^te^{-qs}(\AA f(X_s)-\tilde\AA f(X_s))\dd s=0$ for all $t\in [0,\zeta\land T_\lll^-)$ a.s.-$\PP_x$ for all $\lll< x$ from $I^\circ$ ($\because$ a continuous local martingale of finite variation is constant a.s.). Taking the right-derivative at $0$ in $t$, the continuity of $\AA f-\tilde\AA f$ and the right-continuity of $X$ at $0$ (together with the normality of $X$ and $\zeta>0$ a.s.) entail that $\AA f(x)=\tilde\AA f(x)$. Thus, at least informally, it is suggestive to think of $\AA f$ as being, as it were, attributed to $f$, hence the notation.

With the above preliminaries sorted we have now a description of a  $\Phi_q$, $q\in (0,\infty)$, of  \eqref{eq:scale-identity} in terms a generator-eigenvalue problem. In some sense it is just a restatement of the martingale property, modulo the localizations, but these localizations are non-trivial.

\begin{proposition}\label{proposition:characterization-generator}
Let $q\in [0,\infty)$,  let $\Phi:I\to \mathbb{R}$ be continuous and  bounded on $[a,\infty)$ for all $a\in I$, and 
suppose furthermore that $\lim_\infty\Phi$ exists when it is not the case that  $\{0<\zeta<\infty,X_{\zeta-}=\infty\}$ is negligible. Then \ref{equivalent:i} of Proposition~\ref{proposition:characterization-one} is equivalent to 
\begin{equation}\label{equivalent:ii}
\begin{aligned}
&\text{\eqref{A} holds true with $f=\Phi$, $\AA f=q\Phi\vert_{I^\circ}$ for all $a\in I^\circ$, and}\\
&\text{($\lim_\infty\Phi=0$ or $\{0<\zeta<\infty,X_{\zeta-}=\infty\}$ is negligible).}
\end{aligned}
 \end{equation} 
In fact \eqref{equivalent:ii} implies  Proposition~\ref{proposition:characterization-one}\ref{equivalent:i} [resp. \ref{equivalent:iii}-\ref{equivalent:vi}] even if $\Phi$ is merely right-continuous (in lieu of continuity) [resp. and if $q>0$].
\end{proposition}
\begin{proof}
Set $\Phi(\infty):=\lim_\infty\Phi$ when it is not the case that $\{0<\zeta<\infty,X_{\zeta-}=\infty\}$ is negligible,  $\Phi(\infty):=0$ (arbitrary) otherwise.

If \ref{equivalent:i} holds true then certainly \eqref{A} is verified with $f=\Phi$, $\AA f=q\Phi\vert_{I^\circ}$ for all $a\in I^\circ$, by stopping (taking $T_n=T_n^+$, $n\in \mathbb{N}$). Suppose now the latter holds true.
For all $\lll\in I$, for each $a\in (\lll,\infty)$, there exists a sequence $(T_n)_{n\in \mathbb{N}}$  of $\FF$-stopping times having the indicated properties of \eqref{A} such that  for all $n\in \mathbb{N}$ the process 
$$\left(\Phi(X_{t\land T_a^-\land T_n})e^{-q(t\land T_a^-\land T_n)}\mathbbm{1}_{\{t\land T_a^-\land T_n< \zeta\}}\right)_{t\in [0,\infty)}$$
 is a local martingale, indeed an a.s. bounded martingale. Passing to the limit $n\to\infty$ we get by bounded convergence that the process $(\Phi(X_{t\land T_a^-})e^{-q(t\land T_a^-\land\zeta)}\mathbbm{1}_{\{t\land T_a^-<\zeta\}\cup \{0<\zeta<\infty,X_{\zeta-}=\infty\}})_{t\in [0,\infty)}$ is a martingale (recall that $X=\infty$ on $[\zeta,\infty)$). Passing further to the limit $a\downarrow \lll$ we obtain again by bounded convergence, using the right-continuity of $\Phi$ at $\lll$, the continuity at the minimum of $X$  and \eqref{eq:right-ct}, that the process $(\Phi(X_{t\land T_\lll^-})e^{-q(t\land T_\lll^-\land\zeta)}\mathbbm{1}_{\{t\land T_\lll^-<\zeta\}\cup \{0<\zeta<\infty,X_{\zeta-}=\infty\}})_{t\in [0,\infty)}$ is  a martingale. In particular \eqref{equivalent:ii} implies \ref{equivalent:i}  [resp. \ref{equivalent:iii}-\ref{equivalent:vi}] of Proposition~\ref{proposition:characterization-one}  [resp.  if also $q>0$] (even if $\Phi$ is merely right-continuous).

Suppose now again that \ref{equivalent:i} of Proposition~\ref{proposition:characterization-one} holds true. By the preceding, for all $\lll\in I$, the processes $(\Phi(X_{t\land T_\lll^-})e^{-q(t\land T_\lll^-)}\mathbbm{1}_{\{t\land T_x^-<\zeta\}})_{t\in [0,\infty)}$ and  $(\Phi(X_{t\land T_\lll^-})e^{-q(t\land T_\lll^-\land\zeta)}\mathbbm{1}_{\{t\land T_\lll^-<\zeta\}\cup \{0<\zeta<\infty,X_{\zeta-}=\infty\}})_{t\in [0,\infty)}$ are both martingales. Their difference, $(\Phi(\infty)e^{-q\zeta}\mathbbm{1}_{\{\zeta\leq t,T_\lll^-=\infty\}\cap \{0<\zeta<\infty,X_{\zeta-}=\infty\}})_{t\in [0,\infty)}$ is therefore also a martingale. But this is only possible if either $\Phi(\infty)=0$ or else $T_\lll^-<\zeta$ a.s. on $ \{0<\zeta<\infty,X_{\zeta-}=\infty\}$ for all $\lll\in I$. Inspect the latter.  When $\inf I\notin I$ it clearly means that $ \{0<\zeta<\infty,X_{\zeta-}=\infty\}$ is negligible, just because in that case $T_\lll^-\uparrow\infty$ a.s. as $\lll\downarrow\inf I$. If $\inf I\in I$ it means that $T_{\inf I}^-<\infty$ a.s. on $\{0<\zeta<\infty,X_{\zeta-}=\infty\}=:E$. Suppose, per absurdum, that for some $x_0\in I$, $\PP_{x_0}(E)>0$. By an inductive application of the strong Markov property, on $E$, a.s.-$\PP_{x_0}$, the process $X$ must hit $\inf I$, then attain a level that is $>1+\inf I$, then hit $\inf I$ again and so on and so forth, infinitely often, all in finite time $<\zeta<\infty$. This contradicts  the a.s. existence of left limits of $X$ in $I\cup\{\infty\}$ on the temporal interval $(0,\zeta]$ on the event $\{\zeta<\infty\}$.

Together with what has been established above we infer that also \ref{equivalent:i} of Proposition~\ref{proposition:characterization-one} implies \eqref{equivalent:ii}.
\end{proof}

 \subsection{Time-changed L\'evy processes}
As previously announced, let us  turn here to the determination of an $\AA f$ satisfying \eqref{A}, and to the establishment of  some facts pertaining to the scale functions of \eqref{eq:scale-identity}, both of which we shall do for time-changed, pk-spLp. 

The statement to follow covering all this is certainly known in many particular cases, and (for the most part) perhaps even obvious to the trained eye, but it seems safer to go through the details to preclude the possibility of anything ``adverse'' happening in the general case. It should also be mentioned that time-changing a (nice) Markov process by the inverse of a continuous additive functional yields a (nice) Markov process always \cite[(2.11) \& (4.11)]{ge2011markov}, vis-\`a-vis Proposition~\ref{proposition:fund-class}\ref{fund:i} below, but we would avoid here using the heavy artillery and associated notational overhang of the general theory of Markov processes that is not really needed in reaching this conclusion. 

We suspend temporarily the overarching setting (for the purposes of the next proposition; and only to conclude it a posteriori), albeit $I$ is still a priori an interval of the real line, unbounded above, and $I^\circ$ its interior.  Notation-wise, for any  Laplace exponent $ \gamma$ of a pk-spLp (say of $Z=(Z_v)_{v\in [0,\delta)}$, then $\gamma(s)=\log \mathbb{E}[e^{-s(Z_1-Z_0)};1<\delta]$ for $s\in [0,\infty)$), $\gamma^{-1}:[0,\infty)\to [0,\infty)$ will denote the right-continuous inverse of $\gamma$: 
\begin{equation}\label{eq:notation}
\gamma^{-1}(u):=\inf\{s\in [0,\infty):\gamma(s)> u\},\quad  u\in [0,\infty).
\end{equation}

\begin{proposition}\label{proposition:fund-class}
Let $A:I^\circ\to(0,\infty)$ be locally bounded away from zero and Borel measurable, let $\psi:[0,\infty)\to \mathbb{R}$ be any Laplace exponent of a spectrally positive L\'evy process (for emphasis: no killing, $\psi(0)=0$), and let $p\in [0,\infty)$. We ask 
\begin{enumerate}[(a)]
\item\label{fund:a} in case $\inf I=-\infty$, that $A$ is bounded on $(-\infty,\lll]$ for some $\lll\in \mathbb{R}$;
\item\label{fund:b} in case $\inf I>-\infty$, that  $\inf I\in I$ iff $\int^\infty_{\psi^{-1}(0)+1} \frac{\dd\lambda}{\lambda A(\inf I+ \frac{1}{\lambda})\psi(\lambda)}<\infty$.
\end{enumerate}

Let also $\xi=(\xi_u)_{u\in [0,\eta)}$ be a pk-spLp with Laplace exponent $\psi-p$  under the complete probabilities $(\PP_x)_{x\in \mathbb{R}}$. Thus $e^{u(\psi(\lambda)-p)}=\PP_x[e^{-\lambda(\xi_u-\xi_0)};u<\eta]$ for $\{u,\lambda\}\subset [0,\infty)$, $x\in \mathbb{R}$; in particular (on taking $\lambda=0$), $\eta$ is exponential of rate $p$. We insist that $\xi$ is c\`adl\`ag and has no negative jumps, both with certainty (not just a.s.). Set $$\sigma:=\sigma^-_{\inf I}:=\inf \{u\in [0,\eta):\xi_{u-}\land \xi_u\leq \inf I\}$$ [$\xi_{0-}:=\xi_0$ on $\{\eta>0\}$] and $$F(v):=\int_0^v \frac{\dd u}{A(\xi_u)}\text{ for }v\in [0,\eta\land\sigma].$$  Then, if $\inf I>-\infty$, we have that $\inf I\in I$ iff 
\begin{equation}
\text{$F(\sigma)<\infty$ with positive $\PP_x$-probability on $\{\sigma<\eta\}$ for some $x\in I^\circ$},\tag{C}\label{C}
\end{equation}
 in which case $F(\sigma)<\infty$ a.s.-$\PP_x$ [and also with positive $\PP_x$-probability] on $\{\sigma<\eta\}$ for all $x\in I^\circ$.

 Put further $\tau:=F^{-1}$ on $[0,F(\eta\land\sigma))$ and define $$X_t:=
\begin{cases}
\xi_{\tau_t},& t<F(\eta\land\sigma)\\
\inf I,& t\geq F(\eta\land\sigma)
\end{cases},\quad t\in [0,\zeta),
$$ with $$
\text{$\zeta:=F(\eta)\mathbbm{1}_{\{\sigma=\infty\}}+\infty\mathbbm{1}_{\{\sigma<\eta\}}$
 or $\zeta:=F(\eta\land \sigma)\overset{\text{a.s.}}{=}F(\eta)$ according as to whether $\inf I\in I$ or not.}$$
 We have the following assertions. 
 
\begin{enumerate}[(i)]
\item\label{fund:i} The process $X=(X_t)_{t\in [0,\zeta)}$ in the filtration $\FF:=\overline{\FF^X_+}$ under the probabilities $(\PP_x)_{x\in  I}$ satisfies our standing assumptions, is regular downwards, has $\inf I$ as a trap (i.e. $X\vert_{[T_{\inf I}^-,\infty)}=\inf I$ on $\{T^-_{\inf I}<\zeta\}$) when $\inf I\in I$, and the map  $(I\ni x\mapsto \PP_x(t<\zeta,X_t\in A))$ is even Borel for all $A\in \mathcal{B}_I$. 

\begin{center}
The remainder of the statements of this proposition refer to the process of Item~\ref{fund:i}.
\end{center}
\item\label{fund:ii} If $A$ is locally bounded, then for every continuous $f:I\to\mathbb{R}$, which is bounded on $[x,\infty)$ for all $x\in I$, $C^2$ on $I^\circ$, and for all $a\in I^\circ$, $q\in [0,\infty)$, we have \eqref{A} by taking $\AA f=A\cdot 
\AA^{\psi,p}f$, where $$\AA^{\psi,p}f(x):=-pf(x)+\mathsf{d} f'(x)+\frac{\mathsf{v}^2}{2} f''(x)+\int (f(x+h)-f'(x)h\mathbbm{1}_{(0,1]}(h)-f(x))\mathrm{m}(\dd h),\quad x\in I^\circ,$$ and where $$\psi(z)=-\mathsf{d} z+\frac{\mathsf{v}^2}{2}z^2+\int( e^{-z h}+z h\mathbbm{1}_{(0,1]}(h)-1)\mathsf{m}(\dd h),\quad z\in [0,\infty),$$ uniquely determines the drift $\mathsf{d}\in \mathbb{R}$, volatility coefficient $\mathsf{v}\in [0,\infty)$ and jump (L\'evy) measure $\mathsf{m}$ on $\mathcal{B}_{(0,\infty)}$, satisfying $\mathsf{m}^h[h^2\land 1]<\infty$, of $\psi$.

\item\label{time-changed:i} The process $X$ satisfies \eqref{eq:positive-chances} and for $q=0$ we have  \eqref{eq:scale-identity} with $\Phi_0=e^{-\psi^{-1}(p)\cdot }$.
\item\label{time-changed:iii} Let $\psi^\#:=\psi(\psi^{-1}(p)+\cdot)-p$, so that $\psi^\#$ is the Laplace exponent of a (for emphasis: not killed) spectrally positive L\'evy process that is not drifting to $\infty$ (i.e. $(\psi^\#)^{-1}(0)=0$). Let further $q\in [0,\infty)$. Then up to a strictly positive multiplicative constant the map $\Phi_q$ of \eqref{eq:scale-identity} satisfies $$\Phi_q=\Phi_0\times\Phi_q^\#=e^{-\psi^{-1}(p)\cdot}\times \Phi_q^\#,$$ where $\Phi_q^\#$ is the map of \eqref{eq:scale-identity} but with (ceteris paribus) $\psi^\#$ the Laplace exponent of $\xi$.
\end{enumerate}
\end{proposition}

Intuitively, while in $I^\circ$, the process $X$ should be viewed as being got from $\xi$ by driving along its sample paths with a velocity that is the function $A$ of its position; exiting from $I^\circ$ the process $X$ is stopped at $\inf I$ or relegated to the cemetery $\infty$ (in a natural way) depending on the limiting behaviour.
\begin{proof}
Concerning the statement surrounding \eqref{C} use \ref{fund:b} and \cite[Theorem~2.1]{li2020integral} (strictly speaking it applies directly only for $p=0$, $\inf I=0$ and with $A$ bounded away from zero on $[\lll,\infty)$ for all $\lll\in I^\circ$, but the general case reduces easily to this one) recalling that we are excluding subordinators in the designation of spectrally positive L\'evy processes. When a.s. appears below without further qualification it (still) means a.s.-$\PP_x$ for all $x\in I$ [as per default; we mention it only because there are now also the $\PP_x$, $x\in \mathbb{R}\backslash I$, ``hanging around''].

\ref{fund:i}. The Borel measurability property of $(I\ni x\mapsto \PP_x(t<\zeta,X_t\in A))$ follows (by monotone class) from the apposite property for $\xi$, and the fact that $X$ is just a measurable transformation of $\xi$.

Then we analyze some path properties of $X$ which follow directly from the construction, ``omega-by-omega''. 

The map $F$  is a continuous strictly increasing bijection from $[0,\eta\land \sigma)$ onto $[0,F(\eta\land\sigma))$, just because it is  finite (since $A$ is locally bounded away from zero and $\xi$ is locally bounded away from $\inf I$ and $\infty$ on $[0,\eta\land \sigma)$) and since $A^{-1}>0$. Therefore, $\tau$  time-changes continuously and strictly increasingly the path of $\xi$ on $[0,\eta\land \sigma)$ onto the path of $X$ on $[0,F(\eta\land\sigma))$; additionally, if $\inf I\in I$, then, on $\{F(\sigma)<\infty,\sigma<\eta\}$, we have $\zeta=\infty$, $\lim_{F(\sigma)-}X=\lim_{\sigma-}\xi=\inf I$ and $X\vert_{[F(\sigma),\infty)}=\inf I$. Thus $X$ takes its values   and is c\`adl\`ag in $I$; if $\inf I\in I$, then, when at all (to be precise, on $\{\sigma<\eta,F(\sigma)<\infty\}$ at time $F(\sigma)$, while off $\{\sigma<\eta,F(\sigma)<\infty\}$ we have $\zeta=F(\eta\land \sigma)\overset{\text{a.s.}}{=}F(\eta)$), $X$ reaches $\inf I$ continuously and stays there (in particular $\inf I$ is  a trap). It is clear that $\zeta>0$ a.s.; the probabilities being normal for $\xi$, they are also normal for $X$, and $X$ inherits the regularity downwards property directly from $\xi$. Furthermore, $\lim_{\zeta-}X$ exists in $I\cup \{\infty\}$ a.s. on $\{0<\zeta<\infty\}$ because if $\inf I>-\infty$, then $\lim_\infty\xi=\infty$ a.s. on $\{\sigma=\infty=\eta\}$, while if $\inf I=-\infty$, then $\int_0^\infty\frac{\dd u}{A(\xi_u)}=\infty$ a.s. on $\{\eta=\infty,\lim_\infty\xi\ne\infty\}$ (the latter thanks to \ref{fund:a} and thanks to the asymptotic behaviour of L\'evy processes -- the event in question can only have positive probability when $\xi$ is drifting to $-\infty$ or oscillating, in which case it is easy to argue by the absence of negative jumps, the strong law of large numbers and the strong Markov property that $\xi$ spends an infinite Lebesgue amount of time below the level $\lll$ of \ref{fund:a}). Besides,  a.s. on $\{0<\zeta<\infty,X_{\zeta-}<\infty\}$, we have $\eta<\infty=\sigma$ and $\zeta=F(\eta)$.

Next, let $\GG=(\GG_u)_{u\in [0,\infty)}$ be the right-continuous modification of the natural filtration of $\xi$, relative to which then $\xi$ is adapted and has independent increments, in particular is strong Markov and quasi left-continuous. We extend the process $F$ to $[0,\infty)$ by putting $\xi=\infty$ on $[\eta,\infty)$ and interpreting $A(\partial)=1$ for $\partial\notin I^\circ$, then set $F(\infty):=\lim_\infty F$; we also extend $\tau$ to $[0,\infty)$ as the right-continuous inverse of $F$ and put $\tau(\infty):=\lim_\infty\tau$. The process $F$ is then $\GG$-adapted, nonnegative, nondecreasing, continuous, vanishing at zero, strictly increasing while finite, but possibly reaching $\infty$ in finite time.  By \cite[Proposition~7.9]{kallenberg} $\tau=(\tau_t)_{t\in [0,\infty)}$ is a strictly increasing continuous family of $\GG$-stopping times, the filtration $\HH:=\GG_\tau$   
is right-continuous and for all $\GG$-stopping times $\rho$, $F(\rho)$ is an $\HH$-stopping time, in particular $\zeta$ ($=F(\eta\land\sigma)+\infty\mathbbm{1}_{\{F(\sigma)<F(\eta)\}}$ or $=F(\eta\land \sigma)$ according as to whether $\inf I\in I$ or not) is an $\HH$-stopping time and $X$ is $\HH$-adapted; further, still by (the proof of) the quoted result, for all $u\in [0,\infty)$, $F(u)$ is an $\HH$-stopping time  and 
$\HH_{F_t}\subset \GG_t$ on $\{F_t<\infty\}$ /hence everywhere (!), because $\{F_t=\infty\}=\{\tau_\infty\leq t\}$ and $\HH_\infty\subset \GG_{\tau_\infty}$/ for all $t\in [0,\infty)$ [one does not actually need $F$ to be strictly increasing to get this in \cite[proof of Proposition~7.9]{kallenberg}  just its continuity and the fact that it is strictly increasing while finite]. Besides, $\tau$ in turn is $\HH$-adapted, nonnegative, strictly increasing, continuous, vanishing at zero and $F$ is the right-continuous inverse of $\tau$. Therefore, by \cite[Proposition~7.9]{kallenberg} yet again, if $\rho$ is an $\HH$-stopping, then $\tau_\rho$ is an $\HH_F$-, hence a $\GG$-stopping time, and $\HH_\rho\subset (\HH_F)_{\tau_\rho}\subset \GG_{\tau_\rho}$ on $\{\tau_\rho<\infty\}$ [one must be ever so slightly careful in applying \cite{kallenberg} because there filtrations at stopping times are defined with the ``basic'' $\sigma$-field \cite[p.~120]{kallenberg}, $\GG_\infty$ here, hence the presence of the qualifier ``on $\{\tau_\rho<\infty\}$''] but in consequence of course also $\HH_\rho\subset  \GG_{\tau_\rho}$ just everywhere.

Let now  $(T_n)_{n\in \mathbb{N}}$ be a sequence of $\HH$-stopping times that is nondecreasing to a $T$ a.s. and pick an $x\in I$. Changing the times on a $\PP_x$-negligible set, which will be without loss of generality, we  assume the sequence $(T_n)_{n\in \mathbb{N}}$ is nondecreasing to $T$ (with certainty). 
 Then $(\tau_{T_n})_{n\in \mathbb{N}}$ is a sequence of $\GG$-stopping times that is nondecreasing to $\tau_T$, hence  by quasi left-continuity of $\xi$ in $\GG$, $\lim_{n\to\infty} \xi_{\tau_{T_n}}=\xi_{\tau_T}$ a.s.-$\PP_x$ on $\{\tau_T<\eta\}$.  	Therefore $\lim_{n\to\infty} X_{T_n}=X_{T}$ a.s.-$\PP_x$ on $\{T<F(\eta\land\sigma)\}$, hence on $\{T<\zeta\}$. Thus quasi left-continuity of $X$ in $\HH$ is established. Besides, because $\eta$ is completely inaccessible in $\GG$, $$\PP_x(\tau_{T_n}<\eta\text{ for all }n\in \mathbb{N},\lim_{n\to\infty}\tau_{T_n}=\eta<\infty)=0,$$ therefore $$\PP_x\left(\{T_n<\zeta\text{ for all }n\in \mathbb{N}\}\cap \left\{\lim_{n\to\infty}T_n =\zeta\right\}\cap \{\zeta<\infty,X_{\zeta-}<\infty\}\right)=0.$$ In this manner is inferred the non-announcability of $\zeta$ on $\{0<\zeta<\infty,X_{\zeta-}<\infty\}$ in $\HH$.

Proceeding further, take an $\HH$-stopping time $S$ and a measurable $H:D\to [0,\infty]$ on the space $D$ of c\`adl\`ag paths in $I$ with lifetime and $\inf I$ a trap (when $\inf I\in I$), $x\in I$.  Let $h$ be the measurable map from the space of c\`adl\`ag paths in $\mathbb{R}$ with lifetime and no negative jumps  to the space $D$ for which $X=h(\xi)$; this map has been explicated in the statement of the proposition. Then $\tau_S$ is a $\GG$-stopping time, hence $\PP_x[H(h(\xi_{\tau_S+\cdot}))\vert\GG_{\tau_S}]=\PP_{\xi_{\tau_S}}[H(h(\xi))]$ a.s.-$\PP_x$ on $\{\tau_S<\eta\}$ by the strong Markov property of $\xi$ in $\GG$. Therefore  $\PP_x[H(X_{S+\cdot})\vert\GG_{\tau_S}]=\PP_{X_S}[H(X)]$, a fortiori $\PP_x[H(X_{S+\cdot})\vert\HH_S]=\PP_{X_S}[H(X)]$ a.s.-$\PP_x$ on $\{S<F(\eta\land \sigma)\}$ and trivially on $\{F(\eta\land \sigma)\leq S<\zeta\}$. We have deduced the strong Markov property of $X$ in $\HH$.

To conclude the proof of \ref{fund:i} it remains to take into account Remark~\ref{remark:right-cts-filtration} in order to pass from $\HH$ to  $\FF=\overline{\FF^X_+}$.

\ref{fund:ii}. By It\^o's formula \cite[Theorem~II.5.1]{ikeda1989stochastic} and the L\'evy-It\^o decomposition of the sample paths of a L\'evy process  \cite[Theorem~2.1]{kyprianou} for all $n\in \mathbb{N}$ and $a\in I^\circ$ the process 
$$\left(f(\xi_{u\land \sigma_n^+\land \sigma_a^-})\mathbbm{1}_{\{u\land \sigma_n^+\land \sigma_a^-<\eta\}}-\int_0^{u\land \sigma_n^+\land \sigma_a^-}\AA^{\psi,p}f(\xi_v)\dd v\right)_{u\in [0,\infty)},$$
where $\sigma_a^-:=\inf\{u\in [0,\eta):\xi_{u-}\land \xi_u\leq a\}$ and $\sigma_n^+:=\inf\{u\in [0,\eta): \xi_u> n\}\land\eta$, is a martingale in $\GG$ that is a.s. bounded on each bounded deterministic interval. Time-changing by $\tau$ we infer via optional sampling and the elementary change of variables ``$v=\tau_s$, $s=F(v)$'' that the process
$$\left(f(X_{t\land T_a^-\land T_n})\mathbbm{1}_{\{t\land T_a^-\land T_n^+< \zeta\}}-\int_0^{t\land T_a^-\land T_n^+} A(X_s)\cdot \AA^{\psi,p}  f(X_s)\dd s\right)_{t\in [0,\infty)}$$ is a martingale in $\overline{\HH}$ and a fortiori in $\FF$. Here we use the completion $\overline{\HH}$ instead of $\HH=\GG_\tau$ to ensure that we may indeed pass to  $T_a^-$ from what would otherwise have to be the ``direct copy'' $\tilde T_a^-=\inf\{t\in [0,\zeta):X_{t-}\land X_t\leq a\}$ of $\sigma_a^-$ [initially we have that
$$\left(f(X_{t\land \tilde T_a^-\land T_n})\mathbbm{1}_{\{t\land \tilde T_a^-\land T_n^+< \zeta\}}-\int_0^{t\land \tilde T_a^-\land T_n^+} A(X_s)\cdot \AA^{\psi,p}  f(X_s)\dd s\right)_{t\in [0,\infty)}$$ is a martingale in $\HH$], while local boundedness of $A$ (and therefore of $A\cdot \AA^{\psi,p} f$) is used to justify the application of optional sampling to the above $\GG$-martingale over what are  not necessarily bounded stopping times. Once we take into account the discussion on p.~\pageref{page:discussion}, especially $(\bullet_{2b})$, we indeed get \eqref{A} for all $a\in I^\circ$, all $q\in [0,\infty)$, with the indicated $\AA f$, as desired.

\ref{time-changed:i}. Directly from the construction, taking into account \ref{fund:b} and the statement surrounding \eqref{C}, we  get that for $\lll\in I$ a.s. $\{T_\lll^-<\zeta\}=\{\sigma_\lll^-<\eta\}$. Thus the validity of \eqref{eq:positive-chances} and the fact that \eqref{eq:scale-identity} holds true with $\Phi_0=e^{-\psi^{-1}(p)\cdot }$ for $q=0$ are just a matter of the well-known result for the L\'evy process $\xi$  (see Example~\ref{example:levy-process} below).

\ref{time-changed:iii}. The first statement comes from the classical Esscher transform/exponential change of measure \cite[Eq.~(8.5)]{kyprianou}:
$$\QQ_x\vert_{\mathcal{T}_u}:=(e^{-\psi^{-1}(p)(\xi_u-\xi_0)}\mathbbm{1}_{\{u<\eta\}})\cdot\PP_x\vert_{\mathcal{T}_u},\quad x\in \mathbb{R},\, u\in [0,\infty),$$ where $\mathcal{T}=(\mathcal{T}_u)_{u\in [0,\infty)}$ is  the natural filtration of $\xi$ (we do not want any right-continuity or completeness here). 

In case $\inf I>-\infty$ we note that  \eqref{C} holds for the process $X$ iff it holds for the same process but with $\psi^\#$ the Laplace exponent of $\xi$. The reason for this is that the exponential change of measure gives rise to probabilities that are \emph{locally} equivalent up to lifetime, $$\QQ_x\vert_{\mathcal{T}_u\vert_{\{u<\eta\}}}\sim \PP_x\vert_{\mathcal{T}_u\vert_{\{u<\eta\}}}\text{ for all }u\in [0,\infty),\ x\in \mathbb{R}$$
 and \eqref{C} is something that is determined already by the negligible sets of the probabilities on restriction to  $\cup_{u\in [0,\infty)}\mathcal{T}_u\vert_{\{u<\eta\}}$. 
 
 For the second statement we  assume that  $\xi$ is defined on a sufficiently rich canonical space so that the local measures $\QQ_x$, $x\in I$, can be extended from the algebra $\cup_{u\in [0,\infty)}\mathcal{T}_u$ to the $\sigma$-field $\mathcal{T}_\infty$ as bona fide probabilities \cite[Section~V.4]{parthasarathy} and then completed. Since we are only interested in making a distributional inference this is assumed without loss of generality, and we do it only for convenience (we could instead just as well localize the hitting times and pass to the limit).  Under the new measures $\xi$ is a spectrally positive L\'evy process with Laplace exponent $\psi^\#$ (for emphasis, not killed: $\eta=\infty$ a.s.-$\QQ_x$ for all $x\in I$).  Furthermore, for $\lll\in I$, $T_\lll^-=F(\sigma_\lll^-)$ on $\{T_\lll^-<\zeta\}=\{\sigma_\lll^-<\eta\}$ a.s.-$\PP_x$ and a.s.-$\QQ_x$ for all $x\in I$. Therefore, for $\lll\leq x$ from $I$, using \cite[Corollary~3.11]{kyprianou},
\begin{align*}
\PP_x[e^{-q T_\lll^-};T_\lll^-<\zeta]&=\PP_x[e^{-qF(\sigma_\lll^-)};\sigma_\lll^-<\eta]\\
&=\QQ_x\left[e^{\psi^{-1}(p)(\xi_{\sigma_\lll^-}-\xi_0)}e^{-qF(\sigma_\lll^-)};\sigma_\lll^-<\infty\right]\\
&=e^{-\psi^{-1}(p)(x-\lll)}\QQ_x\left[e^{-qF(\sigma_\lll^-)};\sigma_\lll^-<\infty\right]\\
&=e^{-\psi^{-1}(p)(x-\lll)}\QQ_x[e^{-q T_\lll^-};T_\lll^-<\zeta],
\end{align*}
from which the claim is immediate.
\end{proof}

To conclude this section let us recall the maps $\Phi_q$, $q\in [0,\infty)$, for the classes of processes that we have indicated in  Subsection~\ref{subsection:structure}. The expressions may be found in the pieces of literature that we have already quoted (and will, for the reader's convenience, quote again more precisely presently). They also follow straightforwardly from Propositions~\ref{proposition:characterization-one} and~\ref{proposition:characterization-generator} and the apposite expressions for the generators, which are available in  explicit terms through Proposition~\ref{proposition:fund-class}\ref{fund:ii}, but we shall not belabour the reader with this. We continue to use the notation of \eqref{eq:notation}. 


First, ``plain vanilla'' -- stationary independent increments.

\begin{example}\label{example:levy-process}
In the context of Proposition~\ref{proposition:fund-class} take $I=\mathbb{R}$ and $A\equiv 1$. Thus $X=\xi$ is just a, possibly killed at rate $p$, spectrally positive L\'evy process under the probabilities $(\PP_x)_{x\in \mathbb{R}}$, Laplace exponent $\psi-p$. 
We have
\begin{equation}\label{given:levy}
\PP_x[e^{-qT_\lll^-};T_\lll^-<\zeta]=e^{-\psi^{-1}(p+q)(x-\lll)}
\end{equation}for $\lll\leq x$ from $\mathbb{R}$, $q\in [0,\infty)$. 
It is the well-known result \cite[p.~232, 1st display]{kyprianou} coming out of the exponential martingale of $X$ (which is the (local) martingale of  Proposition~\ref{proposition:characterization-one}\ref{equivalent:i}). 
\end{example}

Second, we touch on the world of self-similar processes.

\begin{example}\label{example:logs-ppsmp}
Fix further an $\alpha\in (0,\infty)$. In the context of Proposition~\ref{proposition:fund-class} take $I=\mathbb{R}$ and $A=e^{\alpha \cdot}$. 
Then $\exp(-X)$ under the probabilities $\QQ_y:=\PP_{-\log y}$, $y\in (0,\infty)$, is the  positive self-similar Markov process of the spectrally negative type with self-similarity index $ \alpha$ associated to $-\xi$  via the Lamperti transform (we view $0$ as a cemetery state for $\exp(-X)$); and positive self-similar Markov processes with $0$ absorbing and no positive jumps,  whose paths are not a.s. nonincreasing, are actually exhausted (in law) by this construction \cite[Section~13.3]{kyprianou}.  More directly, in terms of the process $X$ itself, the self-similarity property is the following: for each $c\in (0,\infty)$, $$\left((X_{c^\alpha t})_{t\in [0,c^{-\alpha}\zeta)}\right)_\star\PP_{x}=\left(X-\log c\right)_\star \PP_{x+\log c},\quad x\in \mathbb{R},$$ which is the form that is more convenient for our purposes (it fits in more nicely with the setting of Subsection~\ref{subsection:setting}). We have
\begin{equation}\label{given:ss}
\PP_x[e^{-qT_\lll^-};T_\lll^-<\zeta]=\frac{\sum_{k=0}^\infty\frac{q^k}{\prod_{l=1}^k\psi(\psi^{-1}(p)+l\alpha)-p}e^{-(\psi^{-1}(p)+\alpha k)x}}{\sum_{k=0}^\infty\frac{q^k}{\prod_{l=1}^k\psi(\psi^{-1}(p)+l\alpha)-p}e^{-(\psi^{-1}(p)+\alpha k)\lll}}
\end{equation}for $\lll\leq x$ from $\mathbb{R}$, $q\in[0,\infty)$ \cite[Theorem~13.10(ii)]{kyprianou}.
\end{example}

Third, branching in continuous space. Separate according to whether zero is hit with positive probability or not. We could amalgamate the two cases but it is cleaner to consider them individually.

\begin{example}\label{example:csbp}
Assume $\int^\infty_{\psi^{-1}(0)+1} \frac{1}{\psi}=\infty$ and take in the context of Proposition~\ref{proposition:fund-class}  $I=(0,\infty)$ and $A=\mathrm{id}_{(0,\infty)}$. Then $X$ is a continuous-state branching process with branching mechanism $\psi-p$ under the probabilities $(\PP_x)_{x\in (0,\infty)}$; and by the (another) Lamperti transform continuous-state branching process that do not hit zero from positive levels and whose paths are not a.s. nondecreasing are actually exhausted (in law) by this construction  \cite[Section~12.1]{kyprianou}. We have
\begin{equation}\label{given:branching}
\PP_x[e^{-qT_\lll^-};T_\lll^-<\zeta]=\frac{\int_{\psi^{-1}(p)}^\infty \frac{\dd z}{\psi(z)-p}\exp\left(-xz+\int_{\theta}^z\frac{q}{\psi-p}\right)}{\int_{\psi^{-1}(p)}^\infty \frac{\dd z}{\psi(z)-p}\exp\left(-\lll z+\int_{\theta}^z\frac{q}{\psi-p}\right)}=\frac{x\int_{\psi^{-1}(p)}^\infty \exp\left(-xz+\int_{\theta}^z\frac{q}{\psi-p}\right)\dd z}{\lll\int_{\psi^{-1}(p)}^\infty \exp\left(-\lll z+\int_{\theta}^z\frac{q}{\psi-p}\right)\dd z},
\end{equation}
for  $\lll\leq x$ from $(0,\infty)$, $q\in (0,\infty)$ [the second expression also for $q=0$], with $\theta\in (\psi^{-1}(p),\infty)$ arbitrary \cite[Theorem~1]{ma}.
\end{example}

\begin{example}\label{example:csbp-2}
Assume $\int^\infty_{\psi^{-1}(0)+1} \frac{1}{\psi}<\infty$ and take in the context of Proposition~\ref{proposition:fund-class}  $I=[0,\infty)$ and $A=\mathrm{id}_{(0,\infty)}$. Then (again) $X$ is a continuous-state branching process with branching mechanism $\psi-p$ under the probabilities $(\PP_x)_{x\in [0,\infty)}$; and by the Lamperti transform continuous-state branching process that hit zero with positive probability and whose paths are not a.s. nondecreasing  are  exhausted (in law) by this construction  \cite[Section~12.1]{kyprianou}. We have
\begin{equation}\label{given:branching-2}
\PP_x[e^{-qT_\lll^-};T_\lll^-<\zeta]=\frac{\int_{\psi^{-1}(p)}^\infty \frac{\dd z}{\psi(z)-p}\exp\left(-xz-\int_z^\infty\frac{q}{\psi-p}\right)}{\int_{\psi^{-1}(p)}^\infty \frac{\dd z}{\psi(z)-p}\exp\left(-\lll z-\int_z^\infty\frac{q}{\psi-p}\right)},
\end{equation}
for  $\lll\leq x$ from $[0,\infty)$, $q\in (0,\infty)$ [for $q=0$ one has the same expression as in the previous example] \cite[Theorem~1]{ma}.
\end{example}

As a check the reader may find it worthwhile to compare the expressions of the preceding examples against the validity of Items~\ref{time-changed:i} and~\ref{time-changed:iii} of Proposition~\ref{proposition:fund-class}.  Other examples of processes having (reasonably) explicit maps $\Phi_q$, $q\in [0,\infty)$, falling under Proposition~\ref{proposition:fund-class}, follow from \cite[Subsection~1.2, Theorem~3.1(ii)]{vidmar_2019}.

\section{First passage time characterizations of processes}\label{section:characterizations}

Now for the promised characterizations of the law of $X$ in terms of the laws of its first passage times. Though we did not list it in Subsection~\ref{subsection:structure} we shall do it also for general killed drifts, in Proposition~\ref{proposition:killed-drifts}.

\subsection{L\'evy world} We begin with the simplest of all processes (in the present context), those with stationary independent increments.
 
 \begin{theorem}\label{theorem:levy-character}
Let $I=\mathbb{R}$. Suppose there exist a $Q\in [0,\infty)$ and a map $\phi:(Q,\infty)\to (0,\infty)$, such that 
 \begin{equation}\label{eq:dirac-levy}
\PP_x[e^{-q T_\lll^-};T_\lll^-<\zeta]=e^{-\phi(q)(x-\lll)}\text{ for all }\lll\leq x\text{ from }\mathbb{R},\text { all }q\in (Q,\infty).
\end{equation}
Then $X$ is a pk-spLp in [i.e. adapted to, with independent increments relative to] the filtration $\FF$ under the probabilities $(\PP_x)_{x\in \mathbb{R}}$, the right-continuous inverse of the Laplace exponent of which restricts to $\phi$ on $(Q,\infty)$. Also, the law of $X$ is uniquely determined by \eqref{eq:dirac-levy}.
 \end{theorem}
 Of course a $Q$ (namely, $0$) and a $\phi$ (namely, $\psi^{-1}(p+\cdot)\vert_{(0,\infty)}$) exist for each pk-spLp of Example~\ref{example:levy-process}, so this is a characterization of this class through the nature of the laws of its first passage times downwards. A consequence of the theorem is that if the laws of the first passage times downwards of $X$ coincide with those of a pk-spLp, then the law of $X$ is that of this pk-spLp. But the assertion of Theorem~\ref{theorem:levy-character} is stronger than that: $\phi$ being a restriction of the right-continuous inverse of the Laplace exponent of some pk-spLp is  inferred, not assumed! The case when $\phi$ is affine is elementary, corresponding to a possibly killed negative drift. The main idea of the proof is to use the martingales of Proposition~\ref{proposition:characterization-one}  to get the property of stationary independent increments  at the level of the Laplace transforms.
 \begin{proof}
By Theorem~\ref{theorem:skip-free} we may and do assume $X$ has no negative jumps a.s.. 

Fix $x\in \mathbb{R}$. Letting $q\to\infty$ in \eqref{eq:dirac-levy} [with any $\lll\in (-\infty,x)$] we get by bounded convergence that $\lim_\infty\phi=\infty$. Besides, we see that $\phi$ is continuous. 

Consider next the running infimum process $\underline{X}$ of $X$. \eqref{eq:dirac-levy} means that for $q\in (Q,\infty)$, $(x-\underline{X}_{(e_q\land \zeta)-})_\star \PP_x=\mathrm{Exp}(\phi(q))$, where $e_q$ is an exponential random time of rate $q$ independent of $X$ to which we grant ourselves access (as we may). Let $\lambda\in [0,\infty)$ and $x\in \mathbb{R}$. Then 
$$qe^{\lambda x}\int_0^\infty \PP_x[e^{-\lambda \underline{X}_{(t\land \zeta)-}}]e^{-qt}\dd t=\PP_x[e^{\lambda (x-\underline{X}_{(e_q\land \zeta)-})}]$$
is finite for all $q\in \phi^{-1}((\lambda,\infty))$; picking one such $q$ we get that $-\underline{X}_{(t\land \zeta)-}$ admits a finite $\PP_x$-exponential moment of order $\lambda$  for each fixed $t\in [0,\infty)$. 

By Proposition~\ref{proposition:characterization-one} we have for all real $t\geq s\geq 0$,
$$\PP_x\left[e^{-\phi(q)X(t\land T_\lll^-)-q(t\land T_\lll^-)}\mathbbm{1}_{\{t\land T_\lll^-<\zeta\}}\vert \FF_s\right]=e^{-\phi(q)X(s\land T_\lll^-)-q(s\land T_\lll^-)}\mathbbm{1}_{\{s\land T_\lll^-<\zeta\}}$$ a.s.-$\PP_x$ for all $\lll\leq x$ from $\mathbb{R}$ and $q\in (Q,\infty)$. Letting $\lll\downarrow -\infty$ we deduce by dominated convergence (using the integrability observation procured just above) that 
$$\PP_x\left[e^{-\phi(q)X(t)-qt}\mathbbm{1}_{\{t<\zeta\}}\vert \FF_s\right]=e^{-\phi(q)X(s)-qs}\mathbbm{1}_{\{s<\zeta\}},$$ i.e. 
\begin{equation}\label{thm:levy-proof}
\PP_x\left[e^{-\phi(q)(X(t)-X(s))}\mathbbm{1}_{\{t<\zeta\}}\vert \FF_s\right]=e^{q(t-s)}\mathbbm{1}_{\{s<\zeta\}}
\end{equation}
 a.s.-$\PP_x$. Taking $\PP_x$-expectation of this with $s=0$, $x=0$ gives $\PP_0[e^{-\phi(q)X(t)};t<\zeta]=e^{qt}$ and plugging it back in renders 
$$\PP_x\left[e^{-\phi(q)(X(t)-X(s))}\mathbbm{1}_{\{t<\zeta\}}\vert \FF_s\right]=\PP_0[e^{-\phi(q)X(t-s)};t-s<\zeta]\mathbbm{1}_{\{s<\zeta\}}$$ a.s.-$\PP_x$. Therefore, for all $A\in \FF_s$ and $\lambda\in \phi((Q,\infty))$, 
$$\PP_x\left[e^{-\lambda(X(t)-X(s))};t<\zeta,A\right]=\PP_0[e^{-\lambda X(t-s)};t-s<\zeta]\PP_x(s<\zeta,A),$$
which we extend by analyticity to $\lambda\in \{\Re>0\}$ and then by continuity to $\lambda\in \{\Re\geq 0\}$, to  $\lambda\in \{\Re=0\}=\ii\mathbb{R}$ in particular. By functional monotone class we infer from this that for all bounded Borel $f:\mathbb{R}\to \mathbb{R}$, 
\begin{align*}
\PP_x\left[f(X(t)-X(s));t<\zeta,A\right]&=\PP_0[f(X(t-s));t-s<\zeta]\PP_x(s<\zeta,A)\\
&=\PP_x\left[f(X(t)-X(s))\mathbbm{1}_{\{t<\zeta\}}\vert s<\zeta\right]\PP_x(s<\zeta,A),
\end{align*}
which means that on $\{s<\zeta\}$, under $\PP_x$, $X(t)-X(s)$ is independent of $\FF_s$ and has the same distribution as $X(t-s)$ under $\PP_0$ (we interpret $\infty-r:=\infty$ no matter what the $r\in \mathbb{R}$). But this is to say precisely that $X$ is a possibly killed L\'evy process in $\FF$ under the probabilities $(\PP_x)_{x\in \mathbb{R}}$. Clearly, from \eqref{eq:dirac-levy}, $X$ cannot be a subordinator. Since   $\PP_0[e^{-\phi(q)X(t)}]=e^{qt}$, it is also true that the right-continuous inverse of the Laplace exponent of $X$ restricts to $\phi$ on $(Q,\infty)$.

The second assertion of the theorem is seen  via analyticity of the Laplace exponent of $X$ on $\{\Re>0\}$, continuity on $\{\Re\geq 0\}$: $\phi$ fixes the values of this Laplace exponent on a set with an accumulation point, hence everywhere (and the Laplace exponent determines the law uniquely).
 \end{proof}
As a consequence we get that the property of the additive and regenerative process $T^-$ having stationary increments --- in the sense that  $(T_\lll^-)_\star \PP_x=(T_{\lll-x}^-)_\star {\PP_0}$ for all $\lll\leq x$ from $\mathbb{R}=I$  --- (we will say: of $T^-$ being a subordinator) is characteristic of L\'evy processes.
 \begin{corollary}
 Let $I=\mathbb{R}$. The following are equivalent. 
 \begin{enumerate}[(i)]
 \item\label{levy:i} $T^-$ is  a subordinator and \eqref{eq:positive-chances} holds true. 
 \item\label{levy:ii} $X$ is a pk-spLp in $\FF$ under the probabilities $(\PP_x)_{x\in \mathbb{R}}$.
 \end{enumerate}
 \end{corollary}
 \begin{question}
 Which (laws of) subordinators are got in this way from pk-spLp? Not all: for one, in order for the subordinator $T^-$ to have a positive killing rate the process $X$ must itself  be killed at a positive rate or be drifting to $\infty$, but in either case the Laplace exponent $\phi$ of $T^-$ (notation as in Theorem~\ref{theorem:levy-character}, $Q=0$) will have $\phi'(0+)<\infty$  (which in general of course need not be the case); furthermore, by the temporal Wiener-Hopf factorization \cite[p.~166, Eq.~(3) \& p.~191, Eq.~(2)]{bertoin} $\phi$ must be a special Bernstein function  \cite[Definition~11.1]{bernstein}. 
 
 \noindent More generally, when there is a.s. absence of negative jumps, keeping in mind Theorem~\ref{theorem:skip-free}, what are the possible laws of the additive and regenerative process $T^-$?
 \end{question}
 \begin{proof}
By Theorem~\ref{theorem:skip-free} we may and do assume $X$ has no negative jumps a.s..  Clearly \ref{levy:ii} implies \ref{levy:i} via the spatial homogeneity of L\'evy processes (and, for \eqref{eq:positive-chances}, because we are excluding subordinators). To see it the other way around, fix $q\in (0,\infty)$, put $$l_q(\lll):=-\log \PP_0[e^{-q T_{-\lll}^-};T_{-\lll}^-<\zeta],\quad \lll\in [0,\infty),$$ and note that $l_q:[0,\infty)\to [0,\infty)$ (finite-valued due to \eqref{eq:positive-chances}) is strictly increasing, additive by stationary increments of $T^-$, left-continuous in virtue of \eqref{eq:right-ct}. By the theory of Cauchy's functional equation it follows that there is a $\phi(q)\in (0,\infty)$ such that $l_q(\lll)=\phi(q)\lll$, $\lll\in [0,\infty)$. Thus (by stationary increments of $T^-$ again) we get  \eqref{eq:dirac-levy} with $Q=0$. It remains to apply Theorem~\ref{theorem:levy-character}.
 \end{proof}
 \subsection{Self-similar processes}
 The  results of the previous subsection are perhaps not really that unexpected. That we may invert also Example~\ref{example:logs-ppsmp} is considerably less obvious to the naked eye.

 \begin{theorem}\label{theorem:self-similar}
Let $I=\mathbb{R}$. Suppose there exist a $Q\in [0,\infty)$, an $\alpha\in (0,\infty)$, a $z_0\in [0,\infty)$,  and a sequence $(a_k)_{k\in \mathbb{N}_0}$ in $(0,\infty)$ with $a_0=1$,
 $$\liminf_{k\to\infty}\frac{a_{k}k^2}{a_{k-1}}>0\text{ and }\limsup_{k\to\infty}\frac{ka_k}{a_{k-1}}<\infty$$
(so that in particular $\sum_{k\in \mathbb{N}_0}a_ku^k<\infty$ for all $u\in [0,\infty)$), such that 
 \begin{equation}\label{eq:discrete-self-similar}
\PP_x[e^{-q T_\lll^-};T_\lll^-<\zeta]=\frac{\sum_{k\in \mathbb{N}_0}a_kq^ke^{-(z_0+\alpha k)x}}{\sum_{k\in \mathbb{N}_0}a_kq^ke^{-(z_0+\alpha k)\lll}}\text{ for all }\lll\leq x\text{ from }\mathbb{R},\text { all }q\in (Q,\infty). 
\end{equation}  
Then, for some $\psi$ and $p$ (and the given $\alpha$), $X$ has the law of the process of Example~\ref{example:logs-ppsmp}; necessarily, $z_0=\psi^{-1}(p)$ and 
$$a_k=\left(\prod_{l=1}^k(\psi(\psi^{-1}(p)+\alpha l)-p)\right)^{-1},\quad k\in \mathbb{N}_0.$$  Also, the law of $X$ is uniquely determined by \eqref{eq:discrete-self-similar}. 
 \end{theorem}
 For sure the process $X$ of Example~\ref{example:logs-ppsmp} is such that the associated $Q$, $\alpha$, $z_0$ and  $(a_k)_{k\in \mathbb{N}_0}$ exist for it. The case when $(\exists d\in (0,\infty):\forall l\in \mathbb{N}:a_l^{-1}=l!d^l)$ corresponds to $X$ being the Lamperti time-change associated to $\xi=$ a possibly killed negative drift. 
  \begin{proof}
 According to the $\liminf$-$\limsup$ assumption $\frac{c^k}{k!^2}\leq a_k\leq \frac{C^k}{ k!}$ for all $k\in \mathbb{N}_0$ for some $\{c,C\}\subset (0,\infty)$. In particular $\sum_{k\in \mathbb{N}_0}a_ku^k<\infty$ for all $u\in [0,\infty)$. Besides, by analytic extension and continuity \eqref{eq:discrete-self-similar} prevails for all $q\in [0,\infty)$.
 
 By Theorem~\ref{theorem:skip-free} we may and do assume $X$ has no negative jumps a.s..

  Consider next the running infimum process $\underline{X}$ of $X$. Let $\lambda\in (0,\infty)$ and $x\in \mathbb{R}$. Then for all $q\in (0,\infty)$, for $e_q\sim\mathrm{Exp}(q)$ independent of $X$ (to which we may grant ourselves access),
\begin{align*}
q\int_0^\infty \PP_x[e^{-\lambda \underline{X}_{(t\land \zeta)-}}]e^{-qt}\dd t&=\PP_x[e^{-\lambda \underline{X}_{(e_q\land \zeta)-}}]=\int_0^\infty \PP_x(-\underline{X}_{(e_q\land \zeta)-}> \log(h)/\lambda)\dd h\\
&=\int_{0}^\infty \PP_x(T_{-\log(h)/\lambda}^-<e_q\land \zeta)\dd h=\int_{0}^\infty \PP_x[e^{-qT_{-\log(h)/\lambda}^-};T_{-\log(h)/\lambda}^-< \zeta]\dd h\\
&=e^{-\lambda x}+\int_{e^{-\lambda x}}^\infty \frac{\sum_{k\in \mathbb{N}_0}a_kq^ke^{-(z_0+\alpha k)x}}{\sum_{k\in \mathbb{N}_0}a_kq^kh^{(z_0+\alpha k)/\lambda}}\dd h<\infty;
\end{align*}
picking one such $q$ we get that $-\underline{X}_{(t\land \zeta)-}$ admits a finite $\PP_x$-exponential moment of order $\lambda$  for each deterministic time $t\in [0,\infty)$.

Our goal will be to show that $X$ enjoys the property of self-similarity. To make this more manageable we first change measure to get, informally speaking, a process with an infinite lifetime.

Taking $q=0$ in \eqref{eq:discrete-self-similar}, according to Proposition~\ref{proposition:characterization-one} the process $(e^{-z_0 X_{t\land T_\lll^-}}\mathbbm{1}_{\{t\land T_\lll^-<\zeta\}})_{t\in [0,\infty)}$ is a martingale for each $\lll\in \mathbb{R}$. Because of the exponential integrability of the running infimum of $X$ established above it follows in the limit $\lll\to-\infty$ that  the process $(e^{-z_0(X_t-X_0)}\mathbbm{1}_{\{t<\zeta\}})_{t\in [0,\infty)}$ is a mean-one martingale. Then introduce the change of measure
\begin{equation*}
\RR_x\vert_{\FF_t}=(e^{-z_0(X_t-x)}\mathbbm{1}_{\{t<\zeta\}})\cdot \PP_x\vert_{\FF_t},\quad x\in \mathbb{R},\, t\in [0,\infty);
\end{equation*}
 in general perhaps the new measures can only be defined locally on each $\FF_t$, $t\in [0,\infty)$, not on $\FF_\infty$, but (for now) it is enough. 
 
 We note that $t<\zeta$ a.s.-$\RR_x$ for all $x\in \mathbb{R}$ and $t\in [0,\infty)$. Further, the negative of the running infimum process of $X$ admits finite $ \RR_x$-exponential moments of all nonnegative orders at all deterministic times. Besides, by optional sampling it follows that
\begin{equation}\label{eq:change-of-measure-pssmp}
\mathbbm{1}_{\{\tau<\infty\}}\cdot \RR_x\vert_{\FF_\tau}=(e^{-z_0(X_\tau-x)}\mathbbm{1}_{\{\tau<\zeta\}})\cdot \PP_x\vert_{\FF_\tau},\quad x\in \mathbb{R},
\end{equation}
for any  bounded $\FF$-stopping time $\tau$ (see \cite[proof of Corollary~3.11]{kyprianou}); later on we will be able to omit the qualification ``bounded''.

Let now $q\in (0,\infty)$, $x\in \mathbb{R}$ and $t\in [0,\infty)$. By Proposition~\ref{proposition:characterization-one} we have that 
$$\RR_x\left[e^{-q(t\land T_\lll^-)}\sum_{k\in \mathbb{N}_0}a_kq^ke^{-\alpha kX(t\land T_\lll^-)}\right]=\sum_{k\in \mathbb{N}_0}a_kq^ke^{-\alpha kx}$$ for all real $\lll\leq x$. In particular, by Fatou, $\RR_x\left[\sum_{k\in \mathbb{N}_0}a_kq^ke^{-\alpha kX_t}\right]<\infty$ on letting $\lll\downarrow -\infty$. Furthermore, expanding the left-hand side of the  equality in the preceding display  into a $q$-power series (via $ \RR_x$-linearity and dominated convergence) we may equate coefficients multiplying each $q^k$, $k\in \mathbb{N}_0$ (due to $\sum_{k\in \mathbb{N}_0}a_ku^k<\infty$ for all $u\in [0,\infty)$). In each of the resulting equalities it is then trivial to pass to the limit $\lll\downarrow -\infty$ by dominated convergence using the exponential integrability of the running infimum of $X$. In turn these equalities recombine back into 
$$\RR_x\left[e^{-qt}\sum_{k\in \mathbb{N}_0}a_kq^ke^{-\alpha kX_t}\right]=\sum_{k\in \mathbb{N}_0}a_kq^ke^{-\alpha kx}$$ by $\RR_x$-linearity and dominated convergence again (this time, because $\RR_x\left[\sum_{k\in \mathbb{N}_0}a_kq^ke^{-\alpha kX_t}\right]<\infty$, which was established above). Reversing the order of argument of \cite[Eq.~(13.57)]{kyprianou} (or just carrying $e^{qt}$ to the other side, expanding into a $q$-power series and comparing coefficients in the preceding display) we get 
\begin{equation}\label{eq:self-similar-moment}
\RR_x[e^{-\alpha nX_t}]=\sum_{k=0}^n\frac{a_{n-k}}{a_n}e^{-(n-k)\alpha x}\frac{t^k}{k!},\quad n\in \mathbb{N}_0.
\end{equation} 

At this point we want the measures $\RR_x$, $x\in \mathbb{R}$, to be defined on the whole of $\FF_\infty$. In order to get existence \cite[Section~V.4]{parthasarathy} of such measures we transfer  the probabilities to a sufficiently rich canonical space (of c\`adl\`ag paths with lifetime) with $\FF=\FF^X$ the natural filtration of $X$ (no completion, no right-continuous modification). The latter is allowed because we are only interested in making a distributional inference on the law of $X$ and because only the distributional fact  \eqref{eq:self-similar-moment} concerning the moments of $X$, together with its strong Markov property and quasi left-continuity  in $\FF_+=\FF_+^X$ have to be carried  over. Now $\zeta=\infty$ a.s.-$\RR_x$ for all $x\in \mathbb{R}$ and \eqref{eq:change-of-measure-pssmp} holds true for all $\FF$-stopping times.

We conclude from \eqref{eq:self-similar-moment} that for each $d\in (0,\infty)$, $s\in (0,\infty)$ and $t\in [0,\infty)$ the law of the random variable $[de^{-X_{d^{-\alpha}t}}]^\alpha$ has under $\mathbb{S}_s:=\RR_{-\log(s)}$ the same (nonnegative, integer) moments as does the law of the random variable $[e^{-X_t}]^\alpha$ under $\mathbb{S}_{ds}$. Besides, we may compute and estimate, for $s\in (0,\infty)$ and  $t\in  [0,\infty)$,
\begin{align*}
\sum_{n=0}^\infty \frac{\lambda^n}{n!}\mathbb{S}_s[e^{-\alpha nX_t}]&=\sum_{n=0}^\infty \frac{\lambda^n}{n!}\sum_{k=0}^n\frac{a_{n-k}}{a_n}s^{\alpha(n-k)}\frac{t^k}{k!}\\
&\leq \sum_{n=0}^\infty \frac{\lambda^n}{n!}\sum_{k=0}^n\frac{C^{n-k}(n!)^2}{(n-k)!c^n}s^{\alpha(n-k)}\frac{t^k}{k!}\\
&=\sum_{n=0}^\infty (\lambda/c)^n\sum_{k=0}^n\frac{n!}{(n-k)!k!}(Cs^{\alpha})^{n-k}t^k\\
&=\sum_{n=0}^\infty \left(\frac{(Cs^\alpha+t)\lambda}{c}\right)^n<\infty\text{ for all $\lambda\in [0,\frac{c}{Cs^\alpha+t})$.}
\end{align*}
By this we are able to infer that the two laws under inspection admit the same finite moment generating function on a neighborhood of zero, therefore are equal. In consequence,  for each $d\in (0,\infty)$, $s\in (0,\infty)$ and $t\in [0,\infty)$ the random variable $de^{-X_{d^{-\alpha}t}}$ has under $\mathbb{S}_s$ the same law as does the random variable $e^{-X_t}$ under $\mathbb{S}_{ds}$. In other words, the transition semigroup of the Markov process $U:=(e^{-X_t})_{t\in [0,\infty)}$ under the probabilities $(\mathbb{S}_{ds})_{s\in (0,\infty)}$ is the same as the transition semigroup of the Markov process $(de^{-X_{d^{-\alpha}t}})_{t\in [0,\infty)}$  under the probabilities $(\mathbb{S}_{s})_{s\in (0,\infty)}$ [recall $X$ has cemetery $\infty$; $e^{-\infty}:=0$], no matter what the $d\in (0,\infty)$. But that means that the law of $U$ under $\mathbb{S}_{ds}$ is the same as the law of $(dU_{d^{-\alpha }t})_{t\in [0,\infty)}$ under $\mathbb{S}_s$ whenever $\{d,s\}\subset (0,\infty)$. Since also $X$ and therefore $U$ is c\`adl\`ag, strong Markov and quasi left-continuous in $\FF_+$, this is to say that $U$ under the probabilities $(\mathbb{S}_s)_{s\in (0,\infty)}$ is  a positive self-similar Markov process with self-similarity index $\alpha$ and an a.s. infinite lifetime [$0$ plays the role of the cemetery state]. 

It now follows from the Lamperti transform \cite[Theorem~13.1(i)]{kyprianou} that the process $\xi=(\xi_u)_{u\in [0,\rho)}$ given through
\begin{align*}
\tau_t&:=\int_0^t\frac{\dd s}{U_s^\alpha}=\int_0^te^{\alpha X_s}\dd s,\quad t\in [0,\zeta],\\
\rho&:= \tau(\zeta),\\
\tau^{-1}(u)&:=\inf\{t\in [0,\zeta):\tau_t>u\},\quad u\in [0,\infty),\\
 \xi_u&:=X_{\tau^{-1}(u)}, \quad u\in [0,\rho),
 \end{align*}
 is, under the probabilities $(\mathbb{R}_x)_{x\in \mathbb{R}}$, in the time-changed filtration $\FF_{\tau^{-1}}$ (even in $(\FF_{+})_{\tau^{-1}}$ but we do not need it), a spectrally positive L\'evy process that does not drift to $\infty$ (in particular $\rho=\infty$ a.s.-$\mathbb{R}_x$ for all $x\in \mathbb{R}$); denote by $\omega$ its Laplace exponent.  Then, for all $x\in \mathbb{R}$, $\{u,v\}\subset [0,\infty)$, $A\in \FF_{\tau^{-1}(u)}$ and $\lambda\in [z_0,\infty)$, from \eqref{eq:change-of-measure-pssmp} applied to $\tau=\tau^{-1}(u+v)$ and subsequently  to $\tau=\tau^{-1}(u)$,  
\begin{align*}
\PP_x[e^{-\lambda (\xi_{u+v}-\xi_u)};A,u+v<\rho]&=e^{-z_0x}e^{z_0x}\PP_x[e^{-z_0\xi_{u+v}}e^{-(\lambda-z_0) (\xi_{u+v}-\xi_u)}e^{z_0\xi_u};A,\tau^{-1}(u+v)<\zeta]\\
&=e^{-z_0x}\mathbb{R}_x[e^{-(\lambda-z_0) (\xi_{u+v}-\xi_u)}e^{z_0 \xi_u};A,u+v<\rho]\\
&=e^{\omega(\lambda-z_0)v}e^{-z_0x}\mathbb{R}_x[e^{z_0 \xi_u};A,u<\rho]\\
&=e^{\omega(\lambda-z_0)v}\PP_x(A,u<\rho).
\end{align*}
By an argument that is essentially the same as the one following \eqref{thm:levy-proof} in the proof of Theorem~\ref{theorem:levy-character} it  gives that $\xi$ is a pk-spLp under the probabilities $(\PP_x)_{x\in \mathbb{R}}$ in the filtration $\FF_{\tau^{-1}}$, whose Laplace exponent $\psi-p$ ($\psi(0)=0$, $p\in[0,\infty)$) restricts to $\omega(\cdot-z_0)$ on $[z_0,\infty)$, in particular necessarily $z_0=\psi^{-1}(p)$. 

By setting 
$$\phi_u:=\int_0^ue^{-\alpha\xi_v}\dd v,\quad u\in [0,\rho],$$
so that $\phi^{-1}=\tau$ on $[0,\zeta)$, we get
\begin{align*}
\zeta&= \phi(\rho),\\
X_t&=\xi_{\phi^{-1}(t)}, \quad t\in [0,\zeta),
 \end{align*}
 which means that $X$ is associated to $\xi$ in the manner of Proposition~\ref{proposition:fund-class} with $I=\mathbb{R}$ and $A=e^{\alpha \cdot}$, i.e. $X$ has the law of the process of Example~\ref{example:logs-ppsmp} for the indicated $\psi$, $p$ and the given $\alpha$. Finally, from a comparison of \eqref{eq:discrete-self-similar} with \eqref{given:ss}  [or of  \eqref{eq:self-similar-moment} with \cite[Eq.~(13.55)]{kyprianou}] we deduce that $a_k=(\prod_{l=1}^k\omega(\alpha l))^{-1}=(\prod_{l=1}^k\psi(\psi^{-1}(p)+\alpha l)-p)^{-1}$, $k\in \mathbb{N}_0$. The proof of the first claim is complete. 

As for the second assertion, that the law of $X$ is uniquely determined by \eqref{eq:discrete-self-similar}, it follows easily from the lemma to follow presently. Indeed one just has to look at  the coefficients in \eqref{eq:discrete-self-similar} to determine, via said lemma, $\omega$, hence $\psi-p$ and thus the law of $X$.
 \end{proof}
 \begin{lemma}
Let $\alpha\in (0,\infty)$. The Laplace exponent $\omega$ of a spectrally positive L\'evy process satisfying $\omega'(0+)\geq 0$ (for emphasis: no killing, i.e. $\omega(0)=0$) is uniquely determined by its values on the lattice $\alpha\mathbb{N}$.
 \end{lemma}
 \begin{proof}
 Because $\omega'(0+)\geq 0$ (and because $\omega(0)=0$) we know that we may write 
 $$\omega(z)=-\gamma z+\frac{\sigma^2}{2}z^2+\int( e^{-z h}+z h-1)\mathsf{m}(\dd h),\quad z\in [0,\infty),$$ for unique: drift $\gamma\in \mathbb{R}$, volatility coefficient $\sigma\in [0,\infty)$ and jump (L\'evy) measure $\mathsf{m}$ on $\mathcal{B}_{(0,\infty)}$ satisfying $\mathsf{m}^h[h^2\land h]<\infty$. We show such $\gamma,\sigma,\mathsf{m}$ are uniquely determined by the sequence $(\omega(\alpha n))_{n\in \mathbb{N}}$. First by dominated convergence we identify $2\lim_{n\to\infty}\frac{\omega(\alpha n)}{(\alpha n)^2}=\sigma^2$, hence may assume $\sigma^2=0$. Then rewrite (via Tonelli)
$$
 \omega(z)=z\left(-\gamma +\int_0^\infty(1- e^{-z h})\mathsf{m}((h,\infty))\dd h\right)=z^2\int_0^\infty e^{-zh}\left(\int_h^\infty\mathsf{m}((g,\infty))\dd g-\gamma\right)\dd h,\quad z\in (0,\infty).$$ Knowing $(\omega(\alpha n))_{n\in \mathbb{N}}$ we therefore know (the finite quantities)
 $$\int_0^1 v^n\left(\int_{-\alpha^{-1} \log(v)}^\infty\mathsf{m}((g,\infty))\dd g-\gamma\right)\dd v,\quad n\in \mathbb{N}_0.$$ By functional monotone class we know $\int_{-\alpha^{-1} \log(v)}^\infty\mathsf{m}((g,\infty))\dd g-\gamma$ for Lebesgue-a.e. hence by continuity every $v\in (0,1]$; $\lim_{v\downarrow 0}$ determines $\gamma$, and then the right-continuous tail of $\mathsf{m}$ (hence $\mathsf{m}$ itself) we get by differentiation from the left in $v\in (0,1)$.
 \end{proof}
 
 \subsection{Branching}
We turn to the characterization of continuous-state branching processes. We separate according to whether zero is hit with positive probability or not: we feel the statement and the line of the proof are sufficiently different for the two cases to warrant such separation.
 
  \begin{theorem}\label{theorem:csbp}
Let $I=(0,\infty)$. Suppose that there exist a  $Q\in [0,\infty)$, a $z_0\in [0,\infty)$, a (then any)  $\theta\in (z_0,\infty)$ and a $C^1$ map $g:[z_0,\infty)\to [0,\infty)$ satisfying $g>0$ on $(z_0,\infty)$, $\lim_\infty g=\infty$, $\int^\infty_\theta \frac{1}{g}=\infty$, $g(z_0)=0$,  such that 
 \begin{equation}\label{eq:reverese-branching}
\PP_x[e^{-q T_\lll^-};T_\lll^-<\zeta]=
\frac{\int_{z_0}^\infty \frac{1}{g(z)}\exp\left(-xz+\int_\theta^z\frac{q}{g}\right)\dd z}{\int_{z_0}^\infty \frac{1}{g(z)}\exp\left(-\lll z+\int_\theta^z\frac{q}{g}\right)\dd z},\quad\text{ for all } \lll\leq x\text{ from }(0,\infty),\,\text{all } q\in (Q,\infty).
\end{equation} Then $X$ has the law of the process $X$ of Example~\ref{example:csbp} for some $p$ and $\psi$; necessarily  $z_0=\psi^{-1}(p)$ and $g=(\psi-p)\vert_{[\psi^{-1}(p),\infty)}$. In particular, the law of $X$ is uniquely determined by \eqref{eq:reverese-branching}.
 \end{theorem}
Clearly for the process of Example~\ref{example:csbp}  the corresponding quantities $Q$, $z_0$, $\theta$ and $g$ do exist.  The proof will essentially boil down to inferring the branching property from the Laplace transforms.
 \begin{proof}
As previously, by Theorem~\ref{theorem:skip-free} we may and do assume $X$ has no negative jumps a.s..  As another preliminary observation notice that actually \eqref{eq:reverese-branching} rewrites into the more compact form
  \begin{equation}\label{eq:converse-branching}
\PP_x[e^{-q T_\lll^-};T_\lll^-<\zeta]=
\frac{x\int_{z_0}^\infty e^{-xz}\exp\left(\int_\theta^z\frac{q}{g}\right)\dd z}{\lll\int_{z_0}^\infty e^{-\lll z}\exp\left(\int_\theta^z\frac{q}{g}\right)\dd z},\quad\text{ for all } \lll\leq x\text{ from }(0,\infty),\,\text{all } q\in (Q,\infty),
\end{equation}
where we have used $\lim_\infty g=0$ and $\int_{z_0}^\theta\frac{1}{g}=\infty$, which follows from the $C^1$ property of $g$ at $z_0$ rendering $0\leq g'_+(z_0)<\infty$ and from $g(z_0)=0$. By analytic extension and continuity we get this equality for all $q\in [0,\infty)$,  put $\Phi_q(x):=x\int_{z_0}^\infty e^{-x z}\exp\left(\int_\theta^z\frac{q}{g}\right)\dd z$ for $x\in (0,\infty)$, $q\in [0,\infty)$, and note that $\Phi_q$ maps $(0,\infty)$ continuously and strictly decreasingly onto $(0,\infty)$ for all $q\in(0,\infty)$, while $\Phi_0$ is bounded continuous.
%
%
 
Next, let $x\in (0,\infty)$ and $q\in (0,\infty)$. For $p\in (q,\infty)$ and $e_p\sim \mathrm{Exp}(p)$  independent of $X$ (to which we may grant ourselves access) we compute
\begin{align*}
\PP_x[\Phi_q(\underline{X}_{(e_p\land \zeta)-})]&=\int_0^\infty\PP_x(\Phi_q(\underline{X}_{(e_p\land \zeta)-})>z)\dd z=\Phi_q(x)+\int_{\Phi_q(x)}^\infty\PP_x(\underline{X}_{(e_p\land \zeta)-}<\Phi_q^{-1}(z))\dd z\\
&=\Phi_q(x)+\int_{\Phi_q(x)}^\infty\PP_x(T^-_{\Phi_q^{-1}(z)}< e_p\land \zeta)\dd z\\
&=\Phi_q(x)+\int_{\Phi_q(x)}^\infty\PP_x\left[e^{-pT^-_{\Phi_q^{-1}(z)}};T^-_{\Phi_q^{-1}(z)}< \zeta\right]\dd z\\
&=\Phi_q(x)+\Phi_p(x)\int_{\Phi_q(x)}^\infty\frac{\dd z}{\Phi_p(\Phi_q^{-1}(z))}\\
&\leq \Phi_q(x)+\Phi_p(x)\int_{\Phi_q(x)}^\infty\frac{\dd z}{z^{\frac{p}{q}}}<\infty,
\end{align*}
because by Jensen
\begin{align*}
\Phi_p(y)&=e^{-yz_0}\int_{z_0}^\infty\exp\left(\int_\theta^z\frac{q}{g}\right)^{\frac{p}{q}}ye^{-y(z-z_0)}\dd z\\
&\geq e^{-yz_0}\left(\int_{z_0}^\infty\exp\left(\int_\theta^z\frac{q}{g}\right)ye^{-y(z-z_0)}\dd z\right)^{\frac{p}{q}}=e^{yz_0(\frac{p}{q}-1)}\Phi_q(y)^{\frac{p}{q}}\geq \Phi_q(y)^{\frac{p}{q}},\quad y\in (0,\infty).
\end{align*}
We conclude that $\PP_x[\Phi_q(\underline{X}_{(t\land \zeta)-})]<\infty$ for all $t\in [0,\infty)$, $x\in (0,\infty)$ and $q\in [0,\infty)$ (it is trivial for $q=0$).

With this integrability property in hand, we may proceed as follows. According to Proposition~\ref{proposition:characterization-one} we have  for all $q\in [0,\infty)$, $t\in[0,\infty)$, all $\lll\leq x$ from $(0,\infty)$, 
$$\PP_x[\Phi_q(X_{t\land T_\lll^-})e^{-q(t\land T_\lll^-)};t\land T_\lll^-<\zeta]=\Phi_q(x).$$ By the mentioned integrability and dominated convergence we may pass to the limit $\lll\downarrow 0$ to get 
$$\PP_x[\Phi_q(X_{t})e^{-qt};t<\zeta]=\Phi_q(x).$$
Setting $$U_t(z;x):=\PP_x[e^{-zX_t};t<\zeta],\quad z\in [0,\infty),$$ it transcribes into
   $$e^{qt}\int_{z_0}^\infty xe^{-zx}\exp\left(\int_\theta^z\frac{q}{g}\right)\dd z=-\int_{z_0}^\infty \partial_z(U_t(z;x))\exp\left(\int_\theta^z\frac{q}{g}\right)\dd z<\infty$$ (remark that $-\partial_z(U_t(z;x))=\PP_x[X_te^{-zX_t};t<\zeta]<\infty$ for $z\in (0,\infty)$),
   and upon integrating both sides $-q\int_0^v\dd t$ and plugging back in, for $v\in [0,\infty)$, 
$$
\int_{z_0}^\infty xe^{-zx}\exp\left(\int_\theta^z\frac{q}{g}\right)\dd z
=-\int_{z_0}^\infty \partial_z(U_v(z;x))\exp\left(\int_\theta^z\frac{q}{g}\right)\dd z+q\int_{z_0}^\infty \int_0^v\partial_z(U_t(z;x))\dd t\exp\left(\int_\theta^z\frac{q}{g}\right)\dd z.
$$
 Elementary integration by parts procures, provided $q>0$,
$$\int_{z_0}^\infty \left(\frac{e^{-zx}}{g(z)}-\int_0^v\partial_z(U_t(z;x))\dd t\right)\exp\left(\int_\theta^z\frac{q}{g}\right)\dd z=\int_{z_0}^\infty U_v(z;x)\frac{1}{g(z)}\exp\left(\int_\theta^z\frac{q}{g}\right)\dd z.$$
(There is a minor non-trivial technical point that must be addressed in effecting per partes on $\int_{z_0}^\infty \partial_z(U_v(z;x))\exp\left(\int_\theta^z\frac{q}{g}\right)\dd z$, namely why $\lim_{z\to\infty}U_v(z;x)\exp(\int_\theta^z\frac{q}{g})=0$. But certainly $\int_{z_0}^\infty U_v(z;x)\frac{q}{g(z)}\exp\left(\int_\theta^z\frac{q}{g}\right)\dd z=\PP_x[\Phi_q(X_t);t<\zeta]<\infty$ so in \label{page:new}
$$\left[U_v(z;x)\exp\left(\int_\theta^{z}\frac{q}{g}\right)\right]_{z=z_0}^{z=n}=\int_{z_0}^n U_v(z;x)\frac{q}{g(z)}\exp\left(\int_\theta^{z}\frac{q}{g}\right)\dd z+\int_{z_0}^n \partial_{z}(U_v(z;x))\exp\left(\int_\theta^z\frac{q}{g}\right)\dd z$$ we may take the limit as $n\to \infty$: on the right-hand side the first term is converging towards the finite limit $\PP_x[\Phi_q(X_t);t<\zeta]$ as $n\to\infty$ by monotone convergence, while the second term is convergent towards the finite limit $\int_{z_0}^\infty \partial_z(U_v(z;x))\exp\left(\int_\theta^z\frac{q}{g}\right)\dd z$; therefore on the left-hand side the limit $\lim_{z\to\infty}U_v(z;x)\exp(\int_\theta^z\frac{q}{g})$ must for sure exist and be finite. Because $\int^\infty_\theta \frac{1}{g}=\infty$ and $q\in [0,\infty)$ is arbitrary it can only be zero.)

Now introduce $G(z):=\int_\theta^z\frac{1}{g}$, $z\in (z_0,\infty)$, so that $G$ maps $(z_0,\infty)$ onto $\mathbb{R}$ in a $C^1$ strictly increasing fashion, and effect the change of variables ``$y=G(z)$'' to get
$$\int_\mathbb{R} \left(e^{-xG^{-1}(y)}-g(G^{-1}(y))\int_0^vU_t'(G^{-1}(y);x)\dd t\right)e^{qy}\dd y=\int_\mathbb{R} U_v(G^{-1}(y);x)e^{qy}\dd y.$$ Fixing a $q_0\in (0,\infty)$ it follows that the (nonnegative, since $U'_t(\cdot;x)\leq 0$, as $U_t(\cdot;x)$ is nonincreasing) measures 
$$ \left(e^{-xG^{-1}(y)}-g(G^{-1}(y))\int_0^vU_t'(G^{-1}(y);x)\dd t\right)e^{q_0y}\dd y,\quad y\in \mathbb{R},$$
and $$ U_v(G^{-1}(y);x)e^{q_0y}\dd y,\quad y\in \mathbb{R},$$
have moment generating functions, which are finite and agree on a neighborhood of zero. Therefore (by analytic extension and functional monotone class) the two measures agree. Therefore 
\begin{equation}\label{csbp:fund}
0< e^{-xz}-g(z)\int_0^v\partial_z(U_t(z;x))\dd t= U_v(z;x)<\infty
\end{equation} for Lebesgue-a.e., whence by continuity every $z\in (z_0,\infty)$. Therefore $[0,\infty)\ni s\mapsto U_s(z;x)$ is continuously differentiable and 
$$-g(z)\partial_zU_v(z;x)= \partial_vU_v(z;x),\quad z\in (z_0,\infty).$$ 
Writing $W(v,z;x):=-\log U_v(z;x)\in (0,\infty)$, $z\in [0,\infty)$, we get for each  $z\in (z_0,\infty)$ that 
$$0=\partial_vW(v,z;x)+g(z)\partial_zW(v,z;x)\text{ and } W(0,z;x)=zx.$$

Let us  briefly argue by the method of characteristics that the solution to the preceding PDE boundary value problem at  fixed $x$ is unique in the class of maps $[0,\infty)\times (z_0,\infty)\to\mathbb{R}$ that are $C^1$ in each coordinate. Consider two solutions $W_1$, $W_2$, and their difference $W:=W_1-W_2$ (we do not bother introducing a new letter). Then for $c\in (z_0,\infty)$  arbitrary, we get
 $$\frac{\dd}{\dd z}W\left(\int_c^z \frac{1}{g},z\right)=0,\quad z\in [c,\infty);$$ therefore $$W\left(\int_c^z \frac{1}{g},z\right)\text{ is constant in }z\in [c,\infty).$$
 Taking $z=c$ we infer that this constant is equal to $W(0,c)=0$. At given $z\in (z_0,\infty)$, as we vary $c\in (z_0,z]$, $\int_c^z \frac{1}{g}$ assumes all values in $[0,\infty)$. We get $W=0$, i.e. $W_1=W_2$. 
 
 Uniqueness of the PDE  boundary value problem established, we see by additivity that $$W(v,z;x_1+x_2)=W(v,z;x_1)+W(v,z;x_2),\quad z\in (z_0,\infty),\, \{x_1,x_2\}\subset (0,\infty).$$In other words, because (finite, on a neighborhood of infinity) Laplace transforms determine laws on $[0,\infty]$, $$(X_v)_\star \PP_{x_1+x_2}=({X_v}_\star\PP_{x_1})\star ({X_v}_\star\PP_{x_2}),\quad  \{x_1,x_2\}\subset (0,\infty).$$ But that means precisely that $X$ under the probabilities $(\PP_x)_{x\in (0,\infty)}$ is a continuous-state branching process (that does not hit $0$, so $0$ is naturally excluded from the state-space, $\infty$ is considered as a cemetery rather than a point in the state space outright, but none of this matters (we may add $0$ and $\infty$ as absorbing states if we like)). So, by the Lamperti transform \cite[Theorem~12.1]{kyprianou}, for some $p$ and $\psi$, $X$ has the law of the process of Example~\ref{example:csbp}. Comparing \eqref{given:branching} with \eqref{eq:converse-branching} we deduce finally  that $z_0=\psi^{-1}(p)$ and $g=\psi-p$ on $[z_0,\infty)$, which concludes the proof. 
 \end{proof}
 
   \begin{theorem}\label{theorem:csbp-2}
Let $I=[0,\infty)$ and let $0$ be absorbing for $X$. Suppose that there exist a  $Q\in [0,\infty)$, a $z_0\in [0,\infty)$, and a $C^1$ map $g:[z_0,\infty)\to [0,\infty)$ satisfying $g>0$ \& $\int^\infty_{\cdot} \frac{1}{g}<\infty$ on $(z_0,\infty)$, $g(z_0)=0$,  such that 
 \begin{equation}\label{eq:reverese-branching-2}
\PP_x[e^{-q T_\lll^-};T_\lll^-<\zeta]=
\frac{\int_{z_0}^\infty \frac{1}{g(z)}\exp\left(-xz-\int_z^\infty\frac{q}{g}\right)\dd z}{\int_{z_0}^\infty \frac{1}{g(z)}\exp\left(-\lll z-\int_z^\infty\frac{q}{g}\right)\dd z},\quad\text{ for all } \lll\leq x \text{ from  }[0,\infty),\,\text{all } q\in (Q,\infty).
\end{equation} Then $X$ has the law of the process  of Example~\ref{example:csbp-2} for some $p$ and $\psi$; necessarily  $z_0=\psi^{-1}(p)$ and $g=(\psi-p)\vert_{[\psi^{-1}(p),\infty)}$. In particular, the law of $X$ is uniquely determined by \eqref{eq:reverese-branching-2}.
 \end{theorem}
Here we have had to assume in addition that zero is absorbing, since the first passage times downwards cannot tell us anything about the behaviour of the process out of zero. As hitherto we also note that for the process of Example~\ref{example:csbp-2} the $Q$, $z_0$ and $g$ as stipulated do exist.
 \begin{proof}
Yet again by Theorem~\ref{theorem:skip-free} we may and do assume $X$ has no negative jumps a.s.. By analytic continuation we may and do assume $Q=0$. Let $q\in (0,\infty)$. Put $$\Phi_q(x):=\int_{z_0}^\infty \frac{1}{g(z)}\exp\left(-xz-\int_z^\infty\frac{q}{g}\right)\dd z,\quad x\in [0,\infty),$$ a  bounded continuous map vanishing at infinity, and note that $$q\Phi_q(x)=\int_{z_0}^\infty xe^{-xz}\exp\left(-\int_z^\infty\frac{q}{g}\right)\dd z,\quad x\in (0,\infty).$$  According to Proposition~\ref{proposition:characterization-one} the process $(\Phi_q(X_t)e^{-q(t\land T_0^-)}\mathbbm{1}_{\{t\land T_0^-<\zeta\}})_{t\in [0,\infty)}$ is a martingale. By Proposition~\ref{proposition:characterization-generator} and the discussion surrounding \eqref{A}, especially $(\bullet_{2b})$, so is  $$\left(\Phi_q(X_{t\land T_a^-\land T_n^+})\mathbbm{1}_{\{t\land T_a^-\land T_n^+<\zeta\}}-\int_0^{t\land T_a^-\land T_n^+}q\Phi_q(X_s)\dd s\right)_{t\in [0,\infty)}$$ for all $a\in (0,\infty)$ and $n\in \mathbb{N}$. Letting $n\to\infty$ and then $a\downarrow 0$ we infer by bounded convergence that $(\Phi_q(X_{t})\mathbbm{1}_{\{t<\zeta\}}-\int_0^{t\land T_0^-\land \zeta}q\Phi_q(X_s)\dd s)_{t\in [0,\infty)}$ is a martingale ($\{t<\zeta\}=\{t\land T_0^-<\zeta\}$ a.s. because $\zeta=\infty$ a.s. on $\{T_0^-<\zeta\}$ since $0$ is absorbing for $X$). Therefore, for $x\in [0,\infty)$, on taking $\PP_x$-expectations at deterministic times $t\in [0,\infty)$,  setting 
$$U_t(z;x):=\PP_x[e^{-zX_t};t<\zeta],\quad z\in [0,\infty),$$
we get via Tonelli $$\int_{z_0}^\infty \left(\frac{e^{-xz}}{g(z)}-\int_0^t\partial_z(U_s(z,x))\dd s\right)\exp\left(-\int_z^\infty\frac{q}{g}\right)\dd z=  \int_{z_0}^\infty \frac{U_t(z;x)}{g(z)}\exp\left(-\int_z^\infty\frac{q}{g}\right)\dd z$$ after some rearrangement and on using $X=0$ on $[T_0^-,\infty)$ a.s. (again). 

Fix $q_0\in (0,\infty)$, any will do. The class of bounded maps $\left\{\exp\left(-\int_\cdot^\infty \frac{p}{g}\right):p\in (0,\infty)\right\}$ is multiplicative and generates the Borel $\sigma$-field on $(z_0,\infty)$. By functional monotone class we deduce that 
$$\int_{z_0}^\infty \left(\frac{e^{-xz}}{g(z)}-\int_0^t\partial_z(U_s(z,x))\dd s\right)\exp\left(-\int_z^\infty\frac{q_0}{g}\right)F(z)\dd z=  \int_{z_0}^\infty \frac{U_t(z;x)}{g(z)}\exp\left(-\int_z^\infty\frac{q_0}{g}\right)F(z)\dd z$$ for all bounded Borel $F:(z_0,\infty)\to \mathbb{R}$. This renders 
\begin{equation*}
\infty >\frac{e^{-xz}}{g(z)}-\int_0^t\partial_z(U_s(z,x))\dd s=\frac{U_t(z;x)}{g(z)}>0
\end{equation*} first for Lebesgue-a.e. and then by continuity for all $z\in (z_0,\infty)$. The remainder of the story is now essentially verbatim the same as in the proof of Theorem~\ref{theorem:csbp} (from \eqref{csbp:fund} onwards), mutatis mutandis.
 \end{proof}
 
 \subsection{Killed drifts}  Another example in which the law of  $T^-$ determines already the law of $X$ is the trivial case of a negative drift with killing, velocity and rate of killing being (reasonable) functions of the position. Its character is completely elementary but let us offer the formal details.

 \begin{example}\label{ex:just-killing}
We assume $I$ is open and let $X$ have the following dynamics. It drifts downwards with deterministic speed, which is the (Lebesgue) measurable function  $v:I\to (0,\infty)$ of its position. Apart from that it is killed at a rate, which is the (Lebesgue) measurable function $\omega: I\to [0,\infty)$ of its position. In order to ensure that this is consistent with our standing assumptions (in the  filtration $\FF=\overline{\FF_+^X}$) and that all the points in $I$ below the starting position are hit with positive probability (viz. \eqref{eq:positive-chances}), we assume that $\frac{1+\omega}{v}$  is locally integrable  on $I$, however that $\int_{\inf I}^x\frac{\dd a}{v(a)}$ is infinite for some (equivalently, all) $x\in I$. Fixing an arbitrary $\theta\in I$ we have then $$V(X_0)-V(X_t)=t,\quad t\in [0,\zeta),$$ where $$V(a):=\int_\theta^a\frac{\dd y}{v(y)},\quad a\in I,$$ and we note that $V$ maps $I$ strictly increasingly onto an open interval $J$ of $\mathbb{R}$ that is unbounded below. Furthermore,
\begin{equation}\label{eq:just-killing}
\PP_x[e^{-qT_\lll^-};T_\lll^-<\zeta]=e^{-\int_0^{V(x)-V(\lll)}\omega(V^{-1}(V(x)-t)) \dd t-q(V(x)-V(\lll))}\text{ for $\lll\leq x$ from $I$, $q\in [0,\infty)$.}
\end{equation}
\end{example} 

 \begin{proposition}\label{proposition:killed-drifts}
 Retain the $v$, $\omega$ of Example~\ref{ex:just-killing} with the measurability-integrability assumptions made on them, also $\theta$, $V$, $J$ are as introduced there. Suppose that \eqref{eq:just-killing} holds true (for a given $X$ satisfying our standing assumptions). Then $X$ has the law of the process  of Example~\ref{ex:just-killing}.
 \end{proposition}
 \begin{proof}
By Theorem~\ref{theorem:skip-free} we may and do, for the final time, assume $X$ has no negative jumps a.s.. To simplify the notation consider the space-transformed process $Y:=V(X)$ under the probabilities $\QQ_y:=\PP_{V^{-1}(y)}$, $y\in J$.  We show that it is a unit negative drift with lifetime having the same law (modulo the $V$-transformation) as that of the process of Example~\ref{ex:just-killing}, which is enough. For the process $Y$ and its first-passage times downwards $S_y^-$, $y\in J$, 
\begin{equation*}
\PP_y[e^{-qS_\yyy^-};S_\yyy^-<\zeta]=\frac{\Phi_q(y)}{\Phi_q(\yyy)}\text{ for $\yyy\leq y$ from $J$, $q\in [0,\infty)$,}
\end{equation*}
where $$\Phi_q(y):=e^{-\int_{\gamma}^{y}\rho(z)\dd z -qy},\quad y\in J,$$ $\rho:=\omega\circ V^{-1}$, $\gamma \in J$ arbitrary, and we remark that $\rho$ is locally integrable on $J$ (by change of variables, because $\frac{\omega}{v}$ is on $I$). According to Proposition~\ref{proposition:characterization-one} (maybe $J$ is bounded above, but it does not really matter (recall the discussion from Subsection~\ref{subsection:setting} on p.~\pageref{unbdd-above})), for each $q\in [0,\infty)$, the process $(\Phi_q(Y_t)e^{-qt}\mathbbm{1}_{\{t<\zeta\}})_{t\in [0,\infty)}$ is a local martingale. Being nonnegative it is a supermartingale. Taking $\QQ_y$-expectation, $y\in J$, at deterministic time $t\in [0,\infty)$ we get 
$$\QQ_y\left[\exp\left(-\int_y^{Y_{t}}\rho(z)\dd z-q(Y_{t}-(y-t))\right);t<\zeta\right]\leq1.$$ Since it is true for arbitrary $q$, on letting $q\uparrow\infty$ we infer that $Y_t\geq y-t$ a.s.-$\QQ_y$ on $\{t<\zeta\}$. Therefore the mentioned local martingale is a.s. bounded on bounded deterministic intervals and so a true martingale.   Taking again $\QQ_y$-expectation, $y\in J$, at deterministic time $t\in [0,\infty)$ we conclude that
$$\int_{[0,\infty)}\exp\left(-\int_y^{y+z-t}\rho(w)\dd w-qz\right) \left((Y_{t}-(y-t))_\star \QQ_y\vert_{\{t<\zeta\}}\right)(\dd z)=1.$$ 
Since again this is true for arbitrary $q$ it follows at once both  that $Y_t=y-t$ a.s.-$\QQ_y$ on $\{t<\zeta\}$ and also that $$\QQ_y(t<\zeta)=\exp\left(-\int_0^{t}\rho(y-s)\dd s\right).$$ Put differently, $Y$ is a unit negative drift killed at rate which is the function $\rho$ of its position.
The proof is complete.
 \end{proof}

\subsection{An open problem} All of the above certainly begs the following

 \begin{question}\label{question}
 	Assume \eqref{eq:positive-chances}. Could it be that, in general, the law of $X$ (resp. of $X^{T_{\inf I}^-}$) is determined by the laws of the first-passage times downwards when $\inf I\notin I$ (resp. $\inf I\in I$)? 
 \end{question}
Given the content of Proposition~\ref{proposition:characterization-one}, restricting to $\inf I\notin I$ and $X$ having no negative jumps a.s.,  we are basically asking whether knowing a certain class of local martingales related to the first passage downwards problem is enough to know the law of the process. Alas, the proofs  above do not appear to help  immediately towards a positive resolution of the problem, because they are very much focused on establishing a property specific to the class in question: stationary independent increments, self-similarity, branching property, deterministic drifting with killing; using the specific explicit forms of the $\Phi_q$, $q\in [0,\infty)$. On the other hand, the author cannot provide any counterexample either, so the matter is left completely open. 

As a concluding point, let us emphasize however that even a general affirmative answer to Question~\ref{question} would not make Theorems~\ref{theorem:levy-character},~\ref{theorem:self-similar} and~\ref{theorem:csbp}-\ref{theorem:csbp-2} mere examples thereof: in these theorems certain \emph{forms} of the $\Phi_q$ for $q$ in a neighborhood of infinity, not their \emph{precise nature} (up to a multiplicative constant at fixed $q$) are shown to be already sufficient to infer that the process belongs to one or another class. 



\bibliographystyle{plain}
\bibliography{Branching}

\begin{thebibliography}{10}

\bibitem{Abrahams1986}
J.~Abrahams.
\newblock A survey of recent progress on level-crossing problems for random
  processes.
\newblock In I.~F. Blake and H.~V. Poor, editors, {\em Communications and
  Networks: A Survey of Recent Advances}, pages 6--25. Springer New York, New
  York, NY, 1986.

\bibitem{arbib}
M.~A. Arbib.
\newblock Hitting and martingale characterizations of one-dimensional
  diffusions.
\newblock {\em Zeitschrift für Wahrscheinlichkeitstheorie und Verwandte
  Gebiete}, 4(3):232--247, 1965.

\bibitem{avram_li_li_2021}
F.~Avram, B.~Li, and S.~Li.
\newblock General drawdown of general tax model in a time-homogeneous {M}arkov
  framework.
\newblock {\em Journal of Applied Probability}, 58(4):1131–1151, 2021.

\bibitem{bertoin}
J.~Bertoin.
\newblock {\em {L\'e}vy Processes}.
\newblock Cambridge Tracts in Mathematics. Cambridge University Press,
  Cambridge, 1996.

\bibitem{b-g-m}
R.~M. Blumenthal, R.~K. Getoor, and H.~P. McKean~Jr.
\newblock {Markov processes with identical hitting distributions}.
\newblock {\em Illinois Journal of Mathematics}, 6(3):402 -- 420, 1962.

\bibitem{bondesson}
L.~Bondesson.
\newblock A characterization of first passage time distributions for random
  walks.
\newblock {\em Stochastic Processes and their Applications}, 39:81--88, 10
  1991.

\bibitem{doney-chaumont}
L.~Chaumont and R.~Doney.
\newblock On distributions determined by their upward, space–time
  {W}iener-{H}opf factor.
\newblock {\em Journal of Theoretical Probability}, 33(2):1011--1033, 2020.

\bibitem{dellacherie}
C.~Dellacherie and P.~A. Meyer.
\newblock {\em Probabilities and Potential}.
\newblock North-Holland mathematics studies. Hermann, 1978.

\bibitem{ma}
X.~Duhalde, C.~Foucart, and M.~Ma.
\newblock On the hitting times of continuous-state branching processes with
  immigration.
\newblock {\em Stochastic Processes and their Applications},
  124(12):4182--4201, 2014.

\bibitem{ethier}
S.~N. Ethier and T.~G. Kurtz.
\newblock {\em Markov Processes: Characterization and Convergence}.
\newblock Wiley Series in Probability and Statistics. Wiley, 2005.

\bibitem{ge2011markov}
R.~K. Getoor and R.~M. Blumenthal.
\newblock {\em Markov Processes and Potential Theory}.
\newblock Academic Press, 1968.

\bibitem{ikeda1989stochastic}
N.~Ikeda and S.~Watanabe.
\newblock {\em Stochastic Differential Equations and Diffusion Processes}.
\newblock Kodansha scientific books. North-Holland, 1989.

\bibitem{kallenberg}
O.~Kallenberg.
\newblock {\em Foundations of Modern Probability}.
\newblock Probability and Its Applications. Springer New York, 2002.

\bibitem{kkr}
A.~Kuznetsov, A.~E. Kyprianou, and V.~Rivero.
\newblock The theory of scale functions for spectrally negative {L\'e}vy
  processes.
\newblock In {\em L{\'e}vy Matters II: Recent Progress in Theory and
  Applications: Fractional L{\'e}vy Fields, and Scale Functions}, volume 2061,
  pages 97--186. Springer Berlin Heidelberg, Berlin, Heidelberg, 2013.

\bibitem{mateusz}
M.~Kwa\'snicki.
\newblock Random walks are determined by their trace on~the positive half-line.
\newblock {\em Annales Henri Lebesgue}, 3:1389--1397, 2020.

\bibitem{kyprianou}
A.~E. Kyprianou.
\newblock {\em Fluctuations of L{\'e}vy Processes with Applications:
  Introductory Lectures}.
\newblock Universitext. Springer Berlin Heidelberg, 2014.

\bibitem{landriault_li_zhang_2017}
D.~Landriault, B.~Li, and H.~Zhang.
\newblock A unified approach for drawdown (drawup) of time-homogeneous {M}arkov
  processes.
\newblock {\em Journal of Applied Probability}, 54(2):603–626, 2017.

\bibitem{li2020integral}
P.~S. Li and X.~Zhou.
\newblock Integral functionals for spectrally positive {L\'e}vy processes.
\newblock {\em Journal of Theoretical Probability}, 2022 (in press).

\bibitem{nobile}
A.~G. Nobile, L.~M. Ricciardi, and L.~Sacerdote.
\newblock A note on first-passage time and some related problems.
\newblock {\em Journal of Applied Probability}, 22(2):346--359, 1985.

\bibitem{parthasarathy}
K.~R. Parthasarathy.
\newblock {\em Probability Measures on Metric Spaces}.
\newblock AMS Chelsea Publishing Series. Academic Press, 1972.

\bibitem{pierre}
P.~Patie.
\newblock Infinite divisibility of solutions to some self-similar
  integro-differential equations and exponential functionals of {L\'e}vy
  processes.
\newblock {\em Annales de l'Institut Henri Poincar{\'e}, Probabilit{\'e}s et
  Statistiques}, 45(3):667--684, 2009.

\bibitem{redner}
S.~Redner.
\newblock {\em A Guide to First-Passage Processes}.
\newblock Cambridge University Press, 2001.

\bibitem{revuz-yor}
D.~Revuz and M.~Yor.
\newblock {\em Continuous Martingales and {B}rownian Motion}.
\newblock Grundlehren der mathematischen Wissenschaften. Springer Berlin
  Heidelberg, third edition, 2005.

\bibitem{bernstein}
R.~L. Schilling, R.~Song, and Z.~Vondra\v{c}ek.
\newblock {\em Bernstein Functions: Theory and Applications}.
\newblock De Gruyter Studies in Mathematics. De Gruyter, 2012.

\bibitem{vidmar_2019}
M.~Vidmar.
\newblock First passage upwards for state-dependent-killed spectrally negative
  {L\'e}vy processes.
\newblock {\em Journal of Applied Probability}, 56(2):472–495, 2019.

\bibitem{vidmar2021continuousstate}
M.~Vidmar.
\newblock Continuous-state branching processes with spectrally positive
  migration.
\newblock 2021.
\newblock arXiv:2107.05102 (to appear in \emph{Probability and Mathematical
  Statistics}).

\bibitem{vidmar2021exit}
M.~Vidmar.
\newblock Exit problems for positive self-similar {M}arkov processes with
  one-sided jumps.
\newblock In A.~Donati-Martin, A.~Lejay, and A.~Rouault, editors, {\em
  S{\'e}minaire de Probabilit{\'e}s LI}, Lecture Notes in Mathematics, pages
  91--115. Springer, 2022.

\bibitem{vidmar-branch}
M.~Vidmar.
\newblock Some harmonic functions for killed {M}arkov branching processes with
  immigration and culling.
\newblock {\em Stochastics: An International Journal of Probability and
  Stochastic Processes}, 94(4):578--601, 2022.

\bibitem{makoto}
M.~Yamazato.
\newblock Characterization of the class of upward first passage time
  distributions of birth and death processes and related results.
\newblock {\em Journal of The Mathematical Society of Japan}, 40(3):477--499,
  1988.

\end{thebibliography}
\end{document}